\documentclass[oneside,11pt,reqno]{amsart}

\usepackage{verbatim,upref,amsxtra,amssymb,amscd,graphicx}

\usepackage{mathtools}

\usepackage{color}

\usepackage[colorlinks=true, pagebackref=false]{hyperref}

\usepackage{varioref}

\usepackage{pgf} \usepackage{tikz} \usetikzlibrary{arrows,automata}

\usepackage{url}

\usepackage{subfig}

\usepackage{ogonek}

\DeclareMathAlphabet\mathscr{U}{eus}{m}{n}
\SetMathAlphabet\mathscr{bold}{U}{eus}{b}{n}
\DeclareMathAlphabet\matheur{U}{eur}{m}{n}
\SetMathAlphabet\matheur{bold}{U}{eur}{b}{n}

\numberwithin{equation}{section}

\newtheorem{theo}{Theorem}[section]
\newtheorem{prop}[theo]{Proposition}
\newtheorem{lemm}[theo]{Lemma}
\newtheorem{coro}[theo]{Corollary}

\theoremstyle{definition}

\newtheorem{defi}[theo]{Definition}
\newtheorem{exam}[theo]{Example}
\newtheorem{exas}[theo]{Examples}

\theoremstyle{remark}

\newtheorem{rema}[theo]{Remark}
\newtheorem{rems}[theo]{Remarks}

\newtheorem{pros}[theo]{Problems}

\begin{document}\allowdisplaybreaks\frenchspacing

%\linenumbers

\setlength{\baselineskip}{17pt}

\title{Permutations of $\mathbb{Z}^d$ with restricted movement}

\author{Klaus Schmidt}

\author{Gabriel Strasser}

\address{Klaus Schmidt: Mathematics Institute, University of Vienna, Oskar-Morgen\-stern-Platz 1, A-1090 Vienna, Austria} \email{klaus.schmidt@univie.ac.at}

\address{Gabriel Strasser: Mathematics Institute, University of Vienna, Oskar-Morgen\-stern-Platz 1, A-1090 Vienna, Austria} \email{gabriel.strasser@univie.ac.at}

\thanks{The first author would like to thank Christian Krattenthaler, Elon Lindenstrauss, and Benjy Weiss for helpful conversations, the Banach Center for hospitality both in B\k{e}dlewo and Warsaw during completion of this article, and the Simons Fundation grant 346300 for IMPAN and the Polish MNiSW 2015-2019 matching fund for support during this period.}

%\dedicatory{}

%\date{}

	\begin{abstract}
We investigate dynamical properties of the set of permutations of $\mathbb{Z}^d$ with \textit{restricted movement}, i.e., permutations $\pi $ of $\mathbb{Z}^d$ such that $\pi (\mathbf{n})-\mathbf{n}$ lies, for every $\mathbf{n}\in \mathbb{Z}^d$, in a prescribed finite set $\mathsf{A}\subset \mathbb{Z}^d$. For $d=1$, such permutations occur, for example, in restricted orbit equivalence (cf., e.g., Boyle and Tomiyama (1998), Kammeyer and Rudolph (1997), or Rudolph (1985)), or in the calculation of determinants of certain bi-infinite multi-diagonal matrices. For $d\ge2$ these sets of permutations provide natural classes of multidimensional shifts of finite type.
	\end{abstract}

\maketitle

\section{Introduction}\label{s:Intro}

Let $d\ge1$, and let $S^\infty (\mathbb{Z}^d)$ be the group of all permutations of the integer lattice $\mathbb{Z}^d$. We fix a finite set $\mathsf{A}\subset \mathbb{Z}^d$ (always assumed to be nonempty) and write $\Pi _\mathsf{A}$ for the set of all permutations $\pi \colon \mathbf{n}\to \pi (\mathbf{n})$ of $\mathbb{Z}^d$ which satisfy that
	\begin{equation}
	\label{eq:PiA}
\omega _\mathbf{n}^{(\pi )}\coloneqq \pi (\mathbf{n}) - \mathbf{n} \in \mathsf{A} \enspace \textup{for every}\enspace \mathbf{n}\in \mathbb{Z}^d.
	\end{equation}
In view of \eqref{eq:PiA} we may regard $\Pi _\mathsf{A}$ as a subset of the space $\prod_{\mathbf{n}\in \mathbb{Z}^d}\,(\mathbf{n}+\mathsf{A})$ which is obviously closed, and hence compact, in the product topology on $\prod_{\mathbf{n}\in \mathbb{Z}^d}\,(\mathbf{n}+\mathsf{A})$.

Every $\pi \in \Pi _\mathsf{A}$ is determined by the point $\omega ^{(\pi )} =(\omega ^{(\pi )}_\mathbf{n})_{\mathbf{n}\in \mathbb{Z}^d}\in \mathsf{A}^{\mathbb{Z}^d}$. Furthermore, the set $\Omega _\mathsf{A}=\{\omega ^{(\pi )}:\pi \in \Pi _\mathsf{A}\}$ is a closed subset of the compact space $\mathsf{A}^{\mathbb{Z}^d}$, and the map $\pi \mapsto \omega ^{(\pi )}$ from $\Pi _\mathsf{A}$ to $\Omega _\mathsf{A}$ is a homeomorphism.

In view of the one-to-one correspondence between $\Pi _\mathsf{A}$ and $\Omega _\mathsf{A}$ it will be convenient to write $\pi ^{(\omega )}\in \Pi _\mathsf{A}$ the for the permutation corresponding to an element $\omega \in \Omega _\mathsf{A}$. Then $\omega =\omega ^{(\pi ^{(\omega )})}$ and $\pi ^{(\omega ^{(\pi )})}=\pi $.

	\begin{prop}
	\label{p:PiA}
Let $\varsigma $ be the $\mathbb{Z}^d$-action on itself by translation, given by $\varsigma ^{\mathbf{m}}(\mathbf{n})= \mathbf{n}+\mathbf{m}$, and let $\sigma \colon \mathbf{m}\rightarrow \sigma ^\mathbf{m}$ be the shift action $(\sigma ^\mathbf{m}\omega )_\mathbf{n}=\omega _{\mathbf{n}+\mathbf{m}}$ of $\mathbb{Z}^d$ on $\mathsf{A}^{\mathbb{Z}^d}$.
	\begin{enumerate}
	\item
For every $\mathbf{m}\in \mathbb{Z}^d$, the set $\Pi _\mathsf{A}\subset S^\infty (\mathbb{Z}^d)$ is invariant under the inner automorphism $\pi \mapsto \textup{Ad}_{\varsigma ^\mathbf{m}}(\pi )=\varsigma ^\mathbf{m}\pi \varsigma ^{-\mathbf{m}}$ of $S^\infty (\mathbb{Z}^d)$.
	\item
For every $\omega \in \Omega _\mathsf{A}$ and $\mathbf{m}\in \mathbb{Z}^d$, $\textup{Ad}_{\varsigma ^\mathbf{m}}(\pi ) = \pi ^{(\sigma ^{-\mathbf{m}}\omega )}$. Hence $\Omega _\mathsf{A}\subset \mathsf{A}^{\mathbb{Z}^d}$ is shift-invariant.
	\item
The topological $\mathbb{Z}^d$-dynamical systems $(\Pi _\mathsf{A},\textup{Ad}_\varsigma )$ and $(\Omega _\mathsf{A},\sigma )$ are topologically conjugate.
	\item
For every $\mathbf{b}\in \mathbb{Z}^d$ we set $\mathsf{A}+\mathbf{b}=\{\mathbf{a}+\mathbf{b}: \mathbf{a}\in \mathsf{A}\}$. Then $\Pi _{\mathsf{A}+\mathbf{b}} = \varsigma ^{\mathbf{b}}\Pi _\mathsf{A}$ and $\Omega _{\mathsf{A}+\mathbf{b}}=\Omega _\mathsf{A}+(\dots ,\mathbf{b},\mathbf{b},\mathbf{b},\dots )$. Furthermore, the systems $(\Omega _\mathsf{A},\sigma )$ and $(\Omega _{\mathsf{A}+\mathbf{b}},\sigma )$ $\bigl($and hence the systems $(\Pi _\mathsf{A},\textup{Ad}_\varsigma )$ and $(\Pi _{\mathsf{A}+\mathbf{b}},\textup{Ad}_\varsigma )\bigr)$ are topologically conjugate.
	\end{enumerate}
	\end{prop}

	\begin{proof}
For every $\pi \in S^\infty (\mathbb{Z}^d)$ and $\mathbf{m}\in \mathbb{Z}^d$, the permutation $\textup{Ad}_{\varsigma ^\mathbf{m}}(\pi )$ satisfies that $\textup{Ad}_{\varsigma ^\mathbf{m}}(\pi )(\mathbf{n})=\pi (\mathbf{n}-\mathbf{m})+\mathbf{m}$. Hence $\textup{Ad}_{\varsigma ^\mathbf{m}}(\pi ^{(\omega )})(\mathbf{n})= \omega _{\mathbf{n}-\mathbf{m}} + \mathbf{n}-\mathbf{m} +\mathbf{m} = \omega _{\mathbf{n}-\mathbf{m}}+\mathbf{n} = (\sigma ^{-\mathbf{m}}\omega )_\mathbf{n} +\mathbf{n} = \pi ^{(\sigma ^{-\mathbf{m}}\omega )}(\mathbf{n})$. Since $\pi (\mathbf{n})-\mathbf{n}\in \mathsf{A}$ for every $\mathbf{n}\in \mathbb{Z}^d$ if and only if $\textup{Ad}_{\varsigma ^\mathbf{m}}(\pi )(\mathbf{n})-\mathbf{n}=\pi (\mathbf{n}-\mathbf{m})-\mathbf{n}+\mathbf{m}\in \mathsf{A}$ for every $\mathbf{m},\mathbf{n}\in \mathbb{Z}^d$, the set $\Pi _\mathsf{A}$ is $\textup{Ad}_\varsigma $-invariant. This proves (1) and (2), and (3) -- (4) are obvious.
	\end{proof}

In the Sections \ref{s:permutations} -- \ref{s:parity} we restrict our attention to the case where $d=1$ and $\mathsf{A}$ is a finite interval in $\mathbb{Z}$, which for simplicity we assume to be of the form $\mathsf{A}_K=\{0,\dots ,K\}$ with $K\ge1$ (cf. Proposition \ref{p:PiA} (4)). For notational economy we set
	\begin{equation}
	\label{eq:PiK}
\Pi _K\coloneqq \Pi _{\mathsf{A}_K},\enspace \textup{and}\enspace \Omega _K \coloneqq \Omega _{\mathsf{A}_K}\subset \mathsf{A}_K^\mathbb{Z}.
	\end{equation}
In the Sections 2 and 3 we prove the following results.
	\begin{itemize}\itemsep=5pt
	\item
For every $\omega \in \Omega _K$ there exists an integer $a(\omega )\in \mathsf{A}_K$ such that
	\begin{displaymath}
\Bigl|\sum\nolimits_{n=m}^{m+N-1}\omega _n-Na(\omega )\Bigr| < K^2
	\end{displaymath}
for every $m\in \mathbb{Z}$ and $N\ge1$ (\eqref{eq:cohomology} and Corollary \ref{c:average}). This integer can be viewed as the average `shift' of $\mathbb{Z}$ imparted by the permutation $\pi ^{(\omega )}\in S^\infty (\mathbb{Z})$.
	\item
The space $\Omega _K$ is a shift of finite type (abbreviated as \textit{SFT}) which is not irreducible. Its irreducible components are given by
	\begin{displaymath}
\Omega _{K,l} = \{\omega \in \Omega _K:a(\omega )=l\}, \enspace l\in \mathsf{A}_K,
	\end{displaymath}
(Proposition \ref{p:mixing}).
	\item
For $0\le l \le K$, the \textit{SFTs} $\Omega _{K,l}$ and $\Omega _{K,K-l}$ are topologically conjugate (Proposition \ref{p:comparisons}).
	\item
$|\Omega _{K,0}|=|\Omega _{K,K}| = 1$, and for $l=1,\dots ,K-1$, the topological entropy $h(\Omega _{K,l})$ satisfies that
	\begin{displaymath}
(1-\tfrac{l}{K})\log\,(l+1)\le h(\Omega _{K,l})\le \log \,(l+1)
	\end{displaymath}
(Theorem \ref{t:entropy}).
	\item
For $0\le l <\frac{K}{2}$, $h(\Omega _{K,l}) \le h(\Omega _{K,l+1})$ (Proposition \ref{p:maximal}).
	\end{itemize}

\smallskip In Section \ref{s:parity} we consider periodic points of the \textit{SFT} $\Omega _K$. If $\omega \in \Omega _K$ is periodic with period $p$, say, then the $p$-tuple
	\begin{displaymath}
\pi ^{(\omega )}_{(m,p)}\coloneqq\bigl(\pi ^{(\omega )}(m)\,(\textup{mod}\,p),\dots ,\pi ^{(\omega )}(m+p-1)\,(\textup{mod}\,p)\bigr)
	\end{displaymath}
is, for every $m\in \mathbb{Z}$, a permutation of $(0,\dots ,p-1)$. What is the parity (or sign) of this permutation? In Theorem \ref{t:parity} we prove that there exists a continuous cocycle $\mathsf{s}\colon \mathbb{Z}\times \Omega _K\longrightarrow \{\pm1\}$ which describes the parities of all these permutations. Together with the function $a\colon \Omega _K\longrightarrow \mathsf{A}_K$ in \eqref{eq:a}, this parity cocycle $\mathsf{s}$ can be used to express the determinants of certain circulant-like matrices appearing in entropy calculations for algebraic actions of the discrete Heisenberg group (cf. \cite[Section 8]{Lind+Schmidt} and Example \ref{e:circulant}).

In Section \ref{s:Z2} we consider permutations of $\mathbb{Z}^d$ with $d\ge2$ and prove the following results.

	\begin{itemize}
	\item
If $\mathsf{A}$ is a finite subset of $\mathbb{Z}^d$, $d\ge2$, then $\Omega _\mathsf{A}$ has positive entropy if and only if $|\mathsf{A}|\ge 3$ (Theorem \ref{t:posent}).
	\item
Let $d\ge2$, and let $\mathsf{A}\subset \mathbb{Z}^d$ be a finite subset. Then $\Omega _\mathsf{A}$ is topologically mixing if and only if $\mathsf{D} = \mathsf{A}-\mathsf{A}$ is not contained in a one-dimensional subspace of $\mathbb{R}^d$ (Theorem \ref{t:topmix}).
	\end{itemize}
Finally, in Section \ref{s:example}, we return to one of the simplest $\mathbb{Z}^2$-\textit{SFT's} arising from permutations of $\mathbb{Z}^2$ with restricted movement, the space $\Omega _\mathsf{A}$ with $\mathsf{A}=\{(0,0),(1,0),\linebreak[0](0,1)\}$. This space had appeared already in Example \ref{e:topmix2}, and we describe its dynamical properties (like entropy and the logarithmic growth-rate of the number of its periodic points) in greater detail.

\section{Permutations of $\mathbb{Z}$ with bounded movement}\label{s:permutations}

We set $d=1$. Fix $K\ge1$, put $\mathsf{A}_K=\{0,\dots ,K\}$, and write, as usual, $\sigma $ instead of $\sigma ^1$ for the shift $(\sigma \omega )_k=\omega _{k+1}$ on $\mathsf{A}^\mathbb{Z}$ (cf. Proposition \ref{p:PiA}). For $\omega =(\omega _n)_{n\in \mathbb{Z}}\in \mathsf{A}_K^\mathbb{Z}$ and $m\in \mathbb{Z}$ we put
	\begin{equation}
	\label{eq:omega-tilde}
\tilde\omega _m=\omega _m+m.
	\end{equation}

In the notation of Section \ref{s:Intro} we set $\Omega _K=\{\omega ^{(\pi )}:\pi \in \Pi _K\}\subset \mathsf{A}_K^\mathbb{Z}$. Then $\Omega _K$ is the subshift of $\mathsf{A}_K^\mathbb{Z}$ defined by the following condition: for every $\omega =(\omega _n)_{n\in \mathbb{Z}}\in \Omega _K$, the map $\pi ^{(\omega )}\colon n \to \tilde{\omega }_n,\;n\in \mathbb{Z}$, is a permutation of $\mathbb{Z}$.

For the following lemma we recall a basic definition: if $A$ is a finite alphabet and $L\ge1$, a subshift $\Omega \subset A^\mathbb{Z}$ is an $L$-step \textit{SFT} if there exists a set $F\subset A^{L+1}$ of \textit{forbidden words} such that $\Omega $ is the set of all sequences $\omega \in A^\mathbb{Z}$ not containing any of the words in $F$.

	\begin{lemm}
	\label{l:subshift}
For every $K\ge1$ the subshift $\Omega _K\subset \mathsf{A}_K^\mathbb{Z}$ has the following properties:
	\begin{enumerate}
	\item
Let $\omega \in \Omega _K$.
	\begin{enumerate}
	\item
For every $m,n\in \mathbb{Z}$, $m\ne n$, $\tilde\omega _m\ne\tilde\omega _n$.
	\item
For every $n\in \mathbb{Z}$, $n\in \{\tilde\omega _{n-K},\dots ,\tilde\omega _{n}\}$.
	\end{enumerate}
	\item
An element $\omega \in \mathsf{A}_K^\mathbb{Z}$ lies in $\Omega _K$ if and only if it satisfies the conditions \textup{(a)} and \textup{(b)} in \textup{(1)}: for every $n\in \mathbb{Z}$, the set $\{\tilde\omega _{n-K},\dots ,\tilde\omega _{n}\}$ has $K+1$ distinct elements and contains $n$. In particular, $\Omega _K$ is a $K$-step \textit{SFT}.
	\item
If $\omega \in \Omega _K$ is periodic with period $p$, then
	\begin{equation}
	\label{eq:permutation}
\pi ^{(\omega )}_{(m,p)}\coloneqq\bigl(\tilde{\omega }_m\,(\textup{mod}\,p),\dots ,\tilde{\omega }_{m+p-1}\,(\textup{mod}\,p)\bigr)
	\end{equation}
is, for every $m\in \mathbb{Z}$, a permutation of $(0,\dots ,p-1)$.
	\end{enumerate}
	\end{lemm}

	\begin{proof}
Obvious.
	\end{proof}

	\begin{exam}
	\label{e:sft}
If $K=1$, $\Omega _1=\{(\dots 0,0,0,\dots ), (\dots ,1,1,1,\dots )\}$. For $K=2$, $\Omega _2$ is the union of the fixed points $(\dots 0,0,0,\dots )$, $(\dots ,2,2,2,\dots )$, and the mixing \textit{SFT} determined by the directed graph
	\begin{center}
\scalebox{.6}{
	\begin{tikzpicture}[->,>=stealth',shorten >=2pt,auto,node distance=25mm, inner sep=2mm, semithick]
\tikzstyle{every state}=[minimum size = 5mm, fill=white,draw=black,text=black] \node[state] (A) {$1$}; \node[state] (B) [below right of=A] {$2$}; \node[state] (C) [below left of=A] {$0$}; \path (A) edge [loop] node {} (A) (A) edge node {} (B) (B) edge node {} (C) (C) edge node {} (A) (C) (C) edge [bend right] node {} (B);
	\end{tikzpicture}}
\vspace{-2mm}
	\end{center}
	\end{exam}

More generally, the following is true.

	\begin{lemm}
	\label{l:components}
\textup{(1)} For every $K\ge1$, the \textit{SFT} $\Omega _K$ is $K$-step, but not \textup{($K-1$)}-step.

\smallskip\textup{(2)} For every $K\ge1$ and $\omega \in \Omega _K$ we put
	\begin{equation}
	\label{eq:a}
a(\omega )=|\{k>0:\tilde{\omega }_{-k}\ge 0\}| = |\{k=1,\dots ,K:\tilde{\omega }_{-k}\ge 0\}|.
	\end{equation}
Then
	\begin{equation}
	\label{eq:subshift}
\Omega _{K,l}=\{\omega \in \Omega _K:a(\omega )=l\}
	\end{equation}
is, for every $l=0,\dots ,K$, a closed, shift-invariant subset of $\Omega _K$ which is a \textup{($K$-1)}-step \textit{SFT}.

\smallskip\textup{(3)} The map $a\colon \Omega _K\longrightarrow \mathsf{A}_K$ defined by \eqref{eq:a} is continuous.
 	\end{lemm}

For the proofs of Lemma \ref{l:components} and Theorem \ref{t:parity}, Figure~\ref{f:figure1} will be useful, where we assume that $p>K \ge 2$.

	\begin{figure}[ht]
\includegraphics[width=80mm]{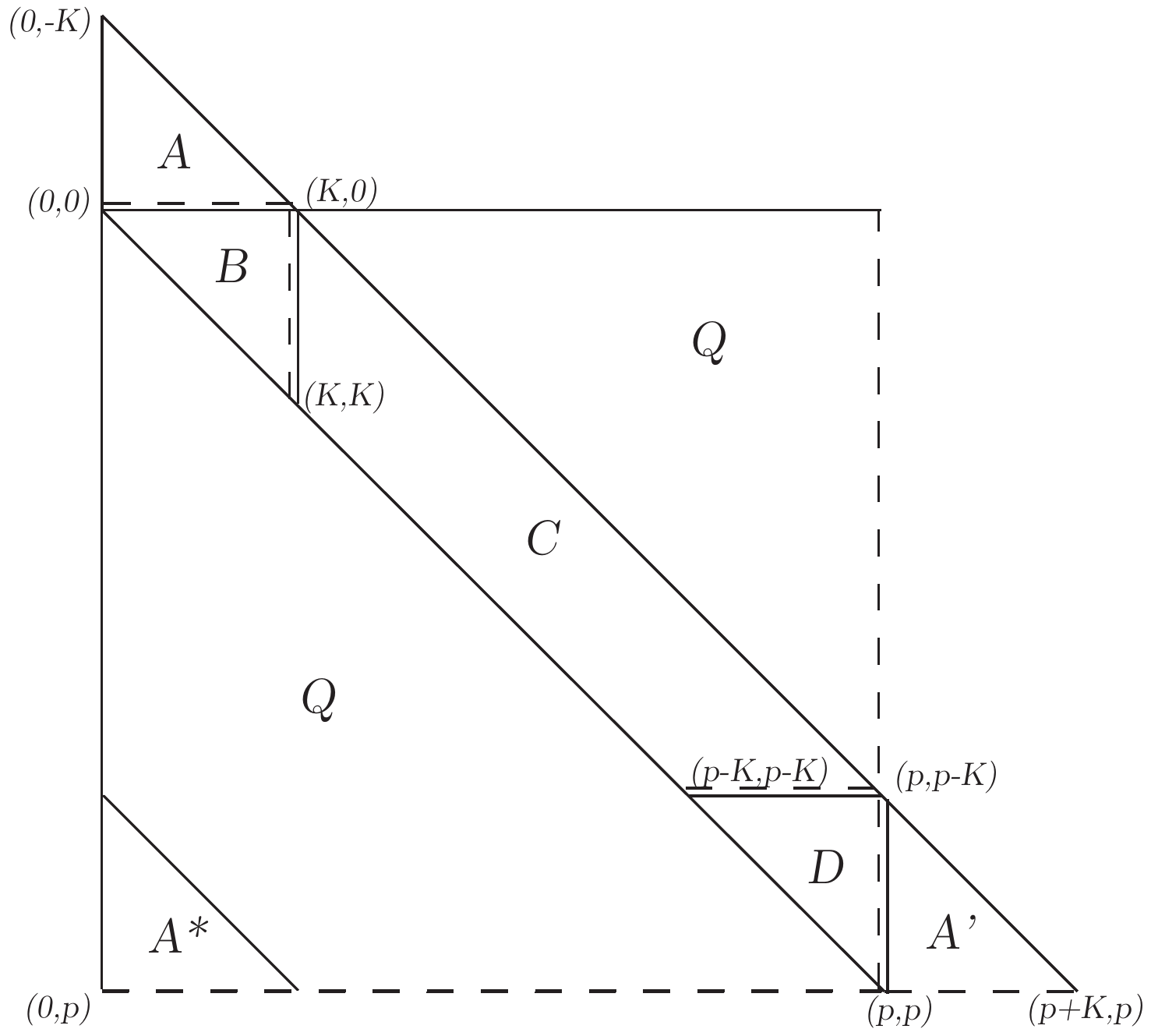}\vspace{-5mm}
	\caption{}\label{f:figure1}
	\end{figure}

Consider the following subsets of $\mathbb{Z}^2$ illustrated in Figure~\ref{f:figure1} (note that the first coordinate in this picture increases as one moves to the right, while the second increases as one moves down):
	\begin{align*}
&\textup{the square} &Q&=\{(k_1,k_2):0\le k_i<p\},
	\\
&\textup{the triangles} & A&=\{(k_1,k_2):0\le k_1 < K, k_1-K \le k_2 < 0\},
	\\
&&B&=\{(k_1,k_2):0\le k_1 < K, 0\le k_2 \le k_1\},
	\\
&&D&=\{(k_1,k_2):p-K\le k_1 < p, p-K\le k_2 \le k_1\},
	\\
&&A'&=\{(k_1,k_2):p\le k_1 < p+K, k_1-K\le k_2 < p\},
	\\
&&A^*&=\{(k_1,k_2):0\le k_1 < K, p-K+k_1\le k_2 < p\},
	\\
&\textup{the trapezoid}&C\hphantom{'}&=\{(k_1,k_2):K\le k_1 <p, k_1-K \le k_2 \le \min \,(k_1, p-K-1)\}.
	\end{align*}
We fix $\omega \in \Omega _K$, set
	\begin{equation}
	\label{eq:S}
S(\omega )=\{(\tilde{\omega }_k,k):k\in \mathbb{Z}\}\subset \mathbb{Z}^2,
	\end{equation}
and denote by $\tilde{A}=A\cap S(\omega )$, $\tilde{B}=B\cap S(\omega )$, \dots , $\tilde{A'}=A'\cap S(\omega )$, $\tilde{C}=C\cap S(\omega )$, and $\tilde{Q}=Q\cap S(\omega )=\tilde{B}\cup \tilde{C}\cup \tilde{D}$ the intersections of these sets with $S(\omega )$ (note that $A^*\cap S(\omega )=\varnothing $).

	\begin{proof}[Proof of Lemma \ref{l:components}]
We start by proving the first assertion in (2). Fix $\omega \in \Omega _K$, define $S(\omega )$ by \eqref{eq:S}, and consider the sets $\tilde{Q}, \tilde{A},\dots ,\tilde{A}',\linebreak[0]\tilde{C}$ defined above. According to the definition of $\Omega _K$, $\{(m,n):n\in \mathbb{Z}\}\cap S(\omega )= \{(m,m-k):0\le k \le K\}\cap S(\omega ) = \{(m,\tilde{\omega }_m)\}$ for every $m\in \mathbb{Z}$. Similarly, $|\{(m,n):m\in \mathbb{Z}\}\cap S(\omega )|= \{(n+k,n):0\le k\le K\}\cap S(\omega ) | = 1$ for every $n\in \mathbb{Z}$. It follows that $|\tilde{A}\cup \tilde{B}| = |\tilde{D}\cup \tilde{A'}| =K$ and $|\tilde{A} \cup \tilde{Q}| = |\tilde{Q}\cup \tilde{A'}|= p$. Since $|\tilde{A}|=a(\omega )$ (cf. \eqref{eq:a}), we obtain that $a(\omega )=|\tilde{A}|=|\tilde{A'}|=a(\sigma ^p\omega )$.

As $p>K$ is arbitrary, $a(\omega )=a(\sigma \omega )$ for every $\omega \in \Omega _K$, which proves that the sets $\Omega _{K,l},\,l=0,\dots ,K$, are shift-invariant.

In order to prove (1) we observe that any $\omega \in \Omega _K$ with $\omega _{-K}=\dots =\omega _{-1} = 0$ satisfies that $a(\omega )=0$ and thus lies in $\Omega _{K,0}$. A point $\omega '\in \Omega _K$ with $\omega _{-K+1}=\dots =\omega _{-1}=0$ and $\omega _0=1$ satisfies $a(\sigma \omega ')=1$ and hence $\omega '\in \Omega _{K,1}$. If $\Omega _K$ were ($K$-1)-step, there would exist a point $\omega ''\in \Omega _K$ with $\omega ''_{-K}=\dots =\omega ''_{-1}=0$ and $\omega ''_0=1$, and hence with $0=a(\omega '')\ne a(\sigma \omega '')=1$. This would contradict the shift-invariance of $a(\cdot )$ shown above.

Having verified (1), we return to (2) by showing that each $\Omega _{K,l}$ is ($K$-1)-step (it is obviously $K$-step). If $l=0$ or $l=K$, $\Omega _{K,l}$ consists of a single fixed point and is therefore $K$-step. If $0<l<K$ we observe that a point $\omega \in \Omega _K$ lies in $\Omega _{K,l}$ if and only if, for every $n\in \mathbb{Z}$, either
	\begin{equation}
	\label{eq:1}\vspace{-4mm}
|\{k=1,\dots ,K-1:\omega _{n-k}>k\}|=l\enspace \textup{and}\enspace \omega _n=0,
	\end{equation}
or
	\begin{equation}
	\label{eq:2}
	\begin{aligned}
|\{k=1,\dots ,K-1&:\omega _{n-k}>k\}|=l-1,\, \omega _n>0,\enspace \textup{and}
	\\
&\qquad \qquad  \; \omega _n\ne \omega _{n-k}-k\enspace \textup{for}\enspace k=1,\dots ,K-1.
	\end{aligned}
	\end{equation}
This proves that $\Omega _{K,l}$ is ($K$-1)-step.

Finally we note that the continuity of the map $a\colon \Omega _K\longrightarrow \mathsf{A}$ claimed in (3) is an immediate consequence of \eqref{eq:a}.
	\end{proof}

We will re-code $\Omega _{K,l}$ as a one-step \textit{SFT} $X_{K,l}$ with alphabet
	\begin{equation}
	\label{eq:SKl}
\mathsf{B}_{K,l} = \bigl\{\mathsf{a}\subset \mathsf{A}_{K-1}:|\mathsf{a}|=l\bigr\},
	\end{equation}
the set of all $l$-element subsets of $\mathsf{A}_{K-1}=\{0,\dots ,K-1\}$. For visualising $X_{K,l}$ we adapt a simile from \cite[Section 2.1]{Redig}.

	\begin{defi}[The \textit{SFT} $X_{K,l}$ -- a strange office]
	\label{d:office}
Consider an office with $K$ desks, numbered from $0$ to $K-1$, and arranged side by side. At each desk there is a clerk who can handle at most one file at any given time. At regular intervals each clerk checks his desk. If he finds a file there he passes it on to his neighbour on the left. If the clerk sitting at the leftmost desk (the desk $0$) finds a file on his desk he carries it over to one of the empty desks, chosen at random (like everybody else he may simultaneously receive a file from his neighbour on the right).

The \textit{SFT} $X_{K,l}$ corresponds to the \textit{modus operandi} of this office if there are a total of $l$ files in circulation.

\smallskip To make this description more formal we view $\mathsf{B}_{K,l}$ as the set of possible positions of $l$ files on the $K$ desks and define a $(\mathsf{B}_{K,l}\negthinspace\times \negthinspace\mathsf{B}_{K,l})$-matrix $\mathsf{M}=\mathsf{M}_{K,l}$ with entries in $\{0,1\}$ by setting, for all $\mathsf{a},\mathsf{b}\in \mathsf{B}_{K,l}$, $\mathsf{M}(\mathsf{a},\mathsf{b})=1$ if and only if one of the following conditions is satisfied:
	\begin{itemize}
	\item[(M1)]
$0\notin \mathsf{a}$ and $\mathsf{b}=\mathsf{a}-1\coloneqq\{a-1:a\in \mathsf{a}\}$ (i.e., there is no file on desk $0$, and each file is moved one step to the left),\label{M1}
	\item[(M2)]
$0\in \mathsf{a}$ and $\mathsf{b}=(\mathsf{a}'-1)\cup \{j\}$ for some $j\in \mathsf{A}_{K-1}\smallsetminus (\mathsf{a}'-1)$, where $\mathsf{a}'=\mathsf{a}\smallsetminus \{0\}$ (i.e., there is a file on desk $0$ which gets dropped on the floor; then all the other files are moved one position to the left, and the file on the floor is placed on an empty desk).\label{M2}
	\end{itemize}
We shall prove that $\Omega _{K,l}$ is conjugate to the \textit{SFT}
	\begin{equation}
	\label{eq:XKl}
X_{K,l}=\bigl\{(\mathsf{a}_n)_{n\in \mathbb{Z}}\in \mathsf{B}_{K,l}^\mathbb{Z}:\mathsf{M}(\mathsf{a}_n,\mathsf{a}_{n+1})=1\enspace \textup{for every}\enspace n\in \mathbb{Z}\bigr\}
	\end{equation}
defined by the transition matrix $\mathsf{M}$.
	\end{defi}

	\begin{prop}
	\label{p:XKl}
Let $\phi _{K,l}\colon \Omega _{K,l}\longrightarrow \mathsf{B}_{K,l}^\mathbb{Z}$ be the map given by
	\begin{equation}
	\label{eq:phiKl}
\phi _{K,l}(\omega )_n = \pi _1(A\cap S(\sigma ^n\omega ))
	\end{equation}
for every $\omega \in \Omega _{K,l}$ and $n\in \mathbb{Z}$, where $A\subset \mathbb{Z}^2$ is the triangle appearing in \textup{Figure 1}~\vpageref{f:figure1}, $S(\sigma ^n\omega )$ is defined in \eqref{eq:S}, and $\pi _1\colon \mathbb{Z}^2\longrightarrow \mathbb{Z}$ is the first coordinate projection. This map has the following properties.
	\begin{enumerate}
	\item
$\phi _{K,l}$ is shift-equivariant and injective.
	\item
$\phi _{K,l}(\Omega _{K,l})=X_{K,l}$.
	\end{enumerate}
	\end{prop}

	\begin{proof}
The shift-equivariance of $\phi _{K,l}$ is obvious from \eqref{eq:phiKl}. The injectivity of $\phi _{K,l}$ follows from \eqref{eq:1} -- \eqref{eq:2}: for every $\omega \in \Omega _{K,l}$ and $n\in \mathbb{Z}$,
	\begin{equation}
	\label{eq:iso1}
\omega _n =0 \enspace \textup{if}\enspace 0\notin \phi _{K,l}(\omega )_n,\quad \textup{and}\quad \omega _n\in (\phi _{K,l}(\omega )_{n+1}+1)\smallsetminus \phi _{K,l}(\omega )_n\enspace \textup{otherwise}.
	\end{equation}
Since this holds for every $n\in \mathbb{Z}$, the sequence $\phi _{K,l}(\omega )$ determines $\omega $, which proves (1).

(2) If a sequence $\omega =(\omega _n)_{n\in \mathbb{Z}}\in \mathsf{A}_K^\mathbb{Z}$ satisfies the conditions \eqref{eq:1} -- \eqref{eq:2} for every $n\in \mathbb{Z}$, then the sets $\phi _{K,l}(\omega )_n,n\in \mathbb{Z}$, satisfy (M1) -- (M2). This shows that $\phi _{K,l}(\Omega _{K,l})\subset X_{K,l}$. Conversely, if $(\mathsf{a}_n)_{n\in \mathbb{Z}}\in X_{K,l}$, and if we set
	\begin{equation}
	\label{eq:iso2}
\omega _n =0 \enspace \textup{if}\enspace 0\notin \mathsf{a}_n,\quad \textup{and}\quad \omega _n\in (\mathsf{a}_{n+1}+1)\smallsetminus \mathsf{a}_n\enspace \textup{otherwise}
	\end{equation}
for every $n\in \mathbb{Z}$, we obtain an element $\omega =(\omega _n)\in \Omega _{K,l}$ with $\phi _{K,l}(\omega )=(\mathsf{a}_n)_{n\in \mathbb{Z}}$. This proves (2).
	\end{proof}

	\begin{prop}
	\label{p:follower-predecessor}
Let\/ $\mathsf{M}$ be the transition matrix of $X_{K,l}$ in \textup{(M1) -- (M2)}. For every state $\mathsf{a}\in\mathsf{B}_{K,l}$ of $X_{K,l}$, the \textit{follower set} $\mathbf{f}(\mathsf{a})=\{\mathsf{b}\in \mathsf{B}_{K,l}:\mathsf{M}(\mathsf{a},\mathsf{b})=1\}$ of $\mathsf{a}$ is given by
	\begin{equation}
	\label{eq:fa1}
\mathbf{f}(\mathsf{a}) = \{\mathsf{b}\in \mathsf{B}_{K,l}: \mathsf{b} \supset (\mathsf{a}-1) \cap A_{K-1}\}
	\end{equation}
and satisfies that
	\begin{equation}
	\label{eq:fa2}
\smash[t]{|\mathbf{f}(\mathsf{a})|=
	\begin{cases}
K-l+1 &\textup{if}\enspace 0\in \mathsf{a},
	\\
1&\textup{otherwise}.
	\end{cases}}
	\end{equation}
The \textit{predecessor set} $\mathbf{p}(\mathsf{a})=\{\mathsf{b}\in \mathsf{B}_{K,l}:\mathsf{M}(\mathsf{b},\mathsf{a})=1\}$ of $\mathsf{a}$ is given by
	\begin{equation}
	\label{eq:pa1}
\mathbf{p}(\mathsf{a})=\{\mathsf{b}\in \mathsf{B}_{K,l}:\mathsf{b}\subset \{0\} \cup (\mathsf{a}+1)\},
	\end{equation}
and satisfies that
	\begin{equation}
	\label{eq:pa2}
\smash[t]{|\mathbf{p}(\mathsf{a})|=
	\begin{cases}
1 &\textup{if}\enspace (K-1) \in \mathsf{a},
	\\
l+1&\textup{otherwise}.
	\end{cases}}
	\end{equation}
	\end{prop}

	\begin{proof}
Let $\mathsf{a}\in \mathsf{B}_{K,l}$. If $0\notin \mathsf{a}$, then $\mathsf{a}-1=\{a-1:a\in \mathsf{a}\}$ is the only follower of $\mathsf{a}$. If $0\in \mathsf{a}$ we put $\mathsf{a}'=\mathsf{a}\smallsetminus \{0\}$ and set $\mathsf{b}=(\mathsf{a}'-1)\cup\{j\}$ for some $j\in \mathsf{A}_K\smallsetminus (\mathsf{a}'-1)$. Clearly, there are $|\mathsf{A}_K\smallsetminus (\mathsf{a}'-1)|=K-l+1$ possibilities for doing this, and every $\mathsf{b}$ obtained in this manner is a follower of $\mathsf{a}$.

For describing the predecessors of $\mathsf{a}$ we first assume that $(K-1)\notin \mathsf{a}$ and set $\mathsf{a}''=(\mathsf{a}+1)\cup\{0\}$. This set has $l+1$ elements, and every $\mathsf{b}\subset \mathsf{A}_K$ obtained by removing one of the elements of $\mathsf{a}''$ is a predecessor of $\mathsf{a}$. Since there are $l+1$ possibilities for doing this, $\mathbf{p}(\mathsf{a})=l+1$. If, on the other hand, $(K-1)\in \mathsf{a}$, then the set $\mathbf{p}(\mathsf{a})$ in \eqref{eq:pa1} has the single element $\mathsf{b}= \{0\} \cup ((\mathsf{a}\smallsetminus \{K-1\})+1)\in \mathsf{B}_{K,l}$.
	\end{proof}

	\begin{prop}
	\label{p:mixing}
For every $K\ge2$ and $l\in \{1,\dots ,K-1\}$, the \textit{SFT} $\Omega _{K,l}$ is irreducible and aperiodic.
	\end{prop}

	\begin{proof}
In view of the isomorphism $\phi _{K,l}\colon \Omega _{K,l}\longrightarrow X_{K,l}$ it suffices to prove the analogous assertion for the \textit{SFT} $X_{K,l}$ in \eqref{eq:XKl}. The state
	\begin{equation}
	\label{eq:e}
\mathsf{e}=\{0,\dots ,l-1\}\in \mathsf{B}_{K,l}
	\end{equation}
satisfies that $|(\mathsf{e}+1)\smallsetminus \mathsf{e}|=1$, so that $\mathsf{e}\in \mathbf{f}(\mathsf{e})$ (and hence $\mathsf{e}\in \mathbf{p}(\mathsf{e})$) by Condition (M2) \vpageref{M2}. Furthermore there obviously exists a path (i.e., a finite sequence of allowed transitions) from every $\mathsf{b}\in \mathsf{B}_{K,l}$ to $\mathsf{e}$. Conversely, if $\mathsf{b}=\{b_1,\dots ,b_l\}$ with $b_1<b_2<\dots <b_l$, there is a path $\mathsf{e}=\{0,\dots ,l-1\}\rightarrow \{0,1,\dots ,l-2,b_1+l-1\}\rightarrow \{0,\dots ,l-3,b_1+l-2,b_2+l-2\}\rightarrow \dots \rightarrow \{0,b_1+1,\dots ,b_{l-1}+1\}\rightarrow \mathsf{b}$ of length $l$. This shows that $X_{K,l}$ is irreducible and aperiodic (since it contains a fixed point).
	\end{proof}

	\begin{prop}
	\label{p:average}
For every $K\ge1$ and $l\in \{0,\dots ,K\}$ there exists a continuous map $b_{K,l}\colon \Omega _{K,l}\longrightarrow \mathbb{Z}$ such that
	\begin{equation}
	\label{eq:cohomology}
\omega _0 = l + b_{K,l}(\omega ) - b_{K,l}(\sigma \omega )
	\end{equation}
for every $\omega \in \Omega _{K,l}$.
	\end{prop}

	\begin{proof}
Fix $K,l$ and $\omega \in \Omega _{K,l}$. We denote by $A=\{(k_1,k_2):0\le k_1 < K, k_1-K \le k_2 < 0\}\subset \mathbb{Z}^2$ the triangle appearing in Figure~\ref{f:figure1} and set, for every $m\in \mathbb{Z}$, $A_m=A+(m,m)$. By \eqref{eq:subshift}, $a(\sigma ^n\omega ) = |A\cap S(\sigma ^n\omega )|=|A_n\cap S(\omega )|=l$ for every $n\in \mathbb{Z}$. Hence $\sum_{m=1}^{N}|A_m\cap S(\omega )|=Nl$ for every $N\ge 1$.

For every $m,n\in \mathbb{Z}$, $(n,\tilde{\omega }_n) \in A_m$  if and only if $\omega _n>0$ and $m=n+1,\dots ,n+\omega _n$. In other words, $\sum_{m\in \mathbb{Z}}1_{A_m}(n,\tilde{\omega }_n) = \sum_{m=1}^K 1_{A_{n+m}}(n,\tilde{\omega }_n) = \omega _n$, where $1_{A_m}\colon \mathbb{Z}^2\longrightarrow \mathbb{Z}$ is the indicator function of $A_m$.

Let $N > K$ (it actually suffices to choose $N>0$, but with $N>K$ one can use Figure~\ref{f:figure1} to see what is going on here). Then
	\begin{align}
	\label{eq:summation}
Nl&=\sum\nolimits_{m=1}^{N}|A_m\cap S(\omega )|=  \sum\nolimits_{{m=1}}^N \sum\nolimits_{{n\in\mathbb{Z}}} 1_{A_m}(n,\tilde{\omega }_n) \notag
	\\
&=\sum\nolimits_{{n\in\mathbb{Z}}} \sum\nolimits_{{m=1}}^N 1_{A_m}(n,\tilde{\omega }_n) \notag
	\\
&=\sum\nolimits_{{n=0}}^{N-1} \;\sum\nolimits_{{m=1}}^K 1_{A_{n+m}}(n,\tilde{\omega }_n) + \sum\nolimits_{{n=-K+1}}^{-1} \;\sum\nolimits_{{m=1}}^{K-1} 1_{A_{m}}(n,\tilde{\omega }_n) \notag
	\\
&\qquad \qquad \qquad \qquad - \sum\nolimits_{{n=N-K+1}}^{N-1} \;\sum\nolimits_{{m=1}}^{K-1} 1_{A_{N+m}}(n,\tilde{\omega }_n)
	\\
&=\sum\nolimits_{{n=0}}^{N-1} \omega _n + \sum\nolimits_{{n=-K+1}}^{-1} \;\sum\nolimits_{{m=1}}^{K-1} 1_{A_{m}}(n,\tilde{\omega }_n)\notag
	\\
&\qquad \qquad \qquad \qquad -\sum\nolimits_{{n=N-K+1}}^{N-1} \;\sum\nolimits_{{m=1}}^{K-1} 1_{A_{N+m}}(n,\tilde{\omega }_n). \notag
	\end{align}
For every $\omega \in \Omega _{K,l}$ we put
	\begin{equation}
	\label{eq:b}
	\begin{aligned}
b_{K,l}(\omega )&=\sum\nolimits_{m=1}^\infty |A\cap A_m \cap S(\omega )| = \sum\nolimits_{m=1}^{K-1}|A\cap A_m \cap S(\omega )|
	\\
&= \sum\nolimits_{{n=-K+1}}^{-1} \;\sum\nolimits_{{m=1}}^{K-1} 1_{A_{m}}(n,\tilde{\omega }_n).
	\end{aligned}
	\end{equation}
Then
	\begin{align*}
b_{K,l}(\sigma ^N\omega )&=\sum\nolimits_{m=1}^{K-1}|A\cap A_m \cap S(\sigma ^N\omega )| = \sum\nolimits_{m=1}^{K-1}|A_N\cap A_{N+m} \cap S(\omega )|
	\\
&=\sum\nolimits_{{n=N-K+1}}^{N-1} \;\sum\nolimits_{{m=1}}^{K-1} 1_{A_{N+m}}(n,\tilde{\omega }_n),
	\end{align*}
and \eqref{eq:summation} shows that
	\begin{equation}
	\label{eq:cohomologyN}
\sum\nolimits_{n=0}^{N-1}\omega _n = Nl + b_{K,l}(\omega ) - b_{K,l}(\sigma ^N\omega )
	\end{equation}
for every $\omega \in \Omega _{K,l}$ and $N\ge0$. By setting $N=1$ in \eqref{eq:cohomologyN} we obtain \eqref{eq:cohomology}.
	\end{proof}

	\begin{coro}
	\label{c:periodic}
Let $K\ge2$, and let $\omega \in \Omega _K$ be a periodic point with period $p$, say. Then $l\coloneqq \frac{1}{p}\sum_{n=0}^{p-1}\omega _n\in \mathsf{A}_K$ and $\omega \in \Omega _{K,l}$.
	\end{coro}

	\begin{coro}
	\label{c:average}
For every $K\ge1$, $l\in \{0,\dots ,K\}$, and $\omega \in \Omega _{K,l}$,
	\begin{equation}
	\label{eq:average}
\lim\nolimits_{N\to\infty }\frac1N \sum\nolimits _{n=0}^{N-1}\omega _n=l.
	\end{equation}
	\end{coro}

	\begin{rems}
	\label{r:Lindner+Strang}
(1) Equation \eqref{eq:cohomology} shows that the cocycle $\boldsymbol{\omega }\colon \mathbb{Z}\times \Omega _{K,l}\longrightarrow \mathbb{Z}$, defined by
	\begin{displaymath}
\boldsymbol{\omega }(n,\omega )=
	\begin{cases}\sum_{k=0}^{n-1}\omega _k &\textup{if}\enspace n>0,
	\\
0&\textup{if}\enspace n=0,
	\\
-\mathbf{\omega }(-n,\sigma ^n\omega )&\textup{if}\enspace n<0,
	\end{cases}
	\end{displaymath}
is cohomologous to the homomorphism $n\mapsto ln$, with transfer function $b_{K,l}$ given by \eqref{eq:b}. In the context of bounded topological orbit equivalence, an analogous formula appears in \cite[Lemma 2.6 and Theorem 2.3 (2)]{BT}.

(2) When combined with Proposition \ref{p:PiA} (4), Corollary \ref{c:average} is equivalent to \cite[Theorem 1]{Lindner+Strang}. In particular, the integer $l=a(\omega )$ in \eqref{eq:a} -- \eqref{eq:subshift} determines, for every $\omega \in \Omega _K$, the position of the `main diagonal' of the bi-infinite permutation matrix associated with $\omega $ (cf. \cite[p. 526]{Lindner+Strang}).
	\end{rems}

We end this section with a few comparisons between the \textit{SFT's} $\Omega _{K,l}$ (or $X_{K,l}$) for different values of $K$ and $l$.

	\begin{prop}
	\label{p:comparisons}
For every $K\ge2$ and $l\in \{1,\dots ,K-1\}$, the following statements hold.
	\begin{enumerate}
	\item
$\Omega _{K,l}\subset \Omega _{K+1,l}$.
	\item
$\Omega _{K,l}+1=\{(\omega _m+1)_{m\in \mathbb{Z}}:\omega \in \Omega _{K,l}\}\subset \Omega _{K+1,l+1}$.
	\item
The \textit{SFT's} $\Omega _{K,l}$ and $\Omega _{K,K-l}$ are topologically conjugate.
	\end{enumerate}
	\end{prop}

	\begin{proof}
The first two assertions are obvious. For the third statement we define a shift-equivariant bijection $\Phi \colon \mathsf{A}_K^\mathbb{Z}\longrightarrow \mathsf{A}_K^\mathbb{Z}$ by setting
	\begin{equation}
	\label{eq:Phi}
\Phi (\omega )_m=K-\omega _{-m}
	\end{equation}
for every $\omega \in \Sigma _{K}$ and $m\in \mathbb{Z}$. It is clear that $\Phi (\Omega _K)=\Omega _K$. In order to check that $\Phi (\Omega _{K,l})=\Omega _{K,K-l}$ we take another look at Figure~\ref{f:figure1}: if $\omega \in \Omega _{K,l}$, then $|\tilde{A}|=|A\cap S(\omega )|=l$. Since $|(A\cap S(\omega ))\cup (B\cap S(\omega ))|=K$, we obtain that $|B\cap S(\omega )|=K-l$. Finally we note that $|A\cap S(\omega )| = |B\cap S(\Phi (\omega ))|$ and $|B\cap S(\omega )| = |A\cap S(\Phi (\omega ))|$, so that $\Phi (\omega )\in \Omega _{K,K-l}$ if and only if $\omega \in \Omega _{K,l}$.
	\end{proof}

	%\begin{comment}
	\begin{rema}
	\label{r:comparisons}
The conjugacy between $\Omega _{K,l}$ and $\Omega _{K,K-l}$ can, of course, also be expressed in terms of the \textit{SFT's} $X_{K,l}$ and $X_{K,K-l}$. For every $\mathsf{a}\in \mathsf{B}_{K,l}$ we put
	\begin{equation}
	\label{eq:a*}
\mathsf{a}^*=\{K-1-j:j\in (\mathsf{A}_{K-1}\smallsetminus \mathsf{a})\}.
	\end{equation}
If $\mathsf{a}\in \mathsf{B}_{K,l}$ then $\mathbf{p}(\mathsf{a}^*) =  \{\mathsf{b}^*:\mathsf{b}\in \mathbf{f}(\mathsf{a})\}=\mathbf{f}(\mathsf{a})^*$ and $\mathbf{f}(\mathsf{a}^*) =  \{\mathsf{b}^*:\mathsf{b}\in \mathbf{p}(\mathsf{a})\}=\mathbf{p}(\mathsf{a})^*$. The corresponding shift-equivariant isomorphism $\Psi_{K,l}: X_{K,l}\longrightarrow X_{K,K-l}$ is given by
	\begin{equation}
	\label{eq:Psi}
\Psi _{K,l}(\mathsf{a} )_m=\mathsf{a}_{-m}^*
	\end{equation}
for every $(\mathsf{a}_n) \in X_{K,l}$ and $m \in \mathbb{Z}$.
	\end{rema}
	%\end{comment}

\section{Entropy}\label{s:entropy}

	\begin{lemm}
	\label{l:entropy}
For every $K,l$ with $0<l<K$, the topological entropy of $\Omega _{K,l}$ satisfies that $0<h(\Omega _{K,l})\le \log\,(l+1)$.
	\end{lemm}

	\begin{proof}
Since $\Omega _{K,l}$ and $X_{K,l}$ are topologically conjugate by Proposition \ref{p:XKl}, their entropies coincide. Clearly, $h(X_{K,l})>0$, since $X_{K,l}$ contains the `diamond' consisting of the paths $\mathsf{e}\rightarrow \mathsf{e}\rightarrow \mathsf{e}$ and $\mathsf{e}\rightarrow \{0,1,\dots ,l-2,l\}\rightarrow \mathsf{e}$ (cf. \eqref{eq:e}), and $h(X_{K,l})\le\log\,(l+1)$ since every state $\mathsf{a}\in \mathsf{B}_{K,l}$ has at most $l+1$ predecessors (cf. \eqref{eq:pa2}).
	\end{proof}

Proposition \ref{p:comparisons} shows that $h(\Omega _{K,l})\le h(\Omega _{K+1,l})$ and $h(\Omega _{K,l})\le h(\Omega _{K+1,l+1})$ whenever $K>l>0$. Our next aim is to investigate $\lim_{K\to\infty }h(\Omega _{K,l})$ for every $l\ge1$.

	\begin{exam}[The \textit{SFT} $\Omega _{K,1}$]
	\label{e:l=1}
\textit{For every $K\ge1$, $h(\Omega _{K,1})$ is equal to $\log\,\beta _K$, where $\beta _K$ is the largest root of the polynomial $f_K(x)=x^{K}-x^{K-1}-\dots -1$. Hence $\lim_{K\to\infty }h(\Omega _{K,1}) = \log 2$.}

\medskip Indeed, the \textit{SFT} $X_{K,1}$ (which is conjugate to $\Omega _{K,1}$) has the form
	\begin{center}
\scalebox{.6}{
	\begin{tikzpicture}[->,>=stealth',shorten >=2pt,auto,node distance=35mm, inner sep=1.5mm, semithick]
\tikzstyle{this state}=[minimum size = 5mm, fill=white,draw=black,text=black];
\node[state] (A) {1};
\tikzstyle{this state}=[minimum size = 5mm, fill=white,draw=black,text=black];
\node[state] (B) [right of=A] {2};
\tikzstyle{this state}=[minimum size = 5mm, fill=white,draw=black,text=black];
\node[state] (C) [right of=B] {3};
\tikzstyle{this state}=[minimum size = 5mm, draw=white,text=black];
\node[this state] (D) [right of=C] {$\dots$};
\tikzstyle{this state}=[minimum size = 5mm, fill=white,draw=black,text=black];
\node[state] (E) [right of=D] {\small{K-1}};
\tikzstyle{this state}=[minimum size = 5mm, fill=white,draw=wblack,text=black];
\node[state] (F) [right of=E] {K};
\path (A) edge [loop] node [above] {1} (A)
(B) edge node [below] {0} (A)
(C) edge node [below] {0}  (B)
(D) edge node [below] {0}  (C)
(E) edge node [below] {0}  (D)
(F) edge node [below] {0}  (E)
(A) edge [bend left = 18] node [above=-1] {2} (B)
(A) edge [bend left = 35] node [above=-1] {3} (C)
(A) edge [bend left = 36] node [above] {K-1} (E)
(A) edge [bend left = 40] node [above] {K} (F);
	\end{tikzpicture}}
\vspace{0mm}
	\end{center}
and is described by the $(K\times K)$-transition matrix
	\begin{displaymath}
P=\left(
	\begin{smallmatrix}
1&1&1&\cdots&1&1&1
	\\
1&0&0&\cdots&0&0&0
	\\
0&1&0&\cdots&0&0&0
	\\
\vdots&&&\ddots&&&\vdots
	\\
0&0&0&\cdots&1&0&0
	\\
0&0&0&\cdots&0&1&0
	\end{smallmatrix}\right)
	\end{displaymath}
with characteristic polynomial $f_K$. The largest root $\beta _K$ of $f_K$ satisfies that $1=\beta _K^{-1}+\dots +\beta _K^{-K}$. As $K\rightarrow \infty $, $\beta _K\rightarrow 2$ and $h(\Omega _{K,1})=h(X_{K,1})=\log \,\beta _K\rightarrow \log 2$ as $K\to\infty $.
	\end{exam}

The following theorem yields an analogous result for arbitrary $l$.

	\begin{theo}
	\label{t:entropy}
For every $K,l$ with $0 < l< K$, $(1-\frac{l}{K})\log\,(l+1)\le h(\Omega _{K,l})\le \log \,(l+1)$. In particular, $\lim_{K\to\infty }h(\Omega _{K,l})=\log\,(l+1)$ for every $l\ge1$.
	\end{theo}

	\begin{proof}
Since $\Omega _{K,l}$ is topologically conjugate to $X_{K,l}$ by Proposition \ref{p:XKl}, it will suffice to prove the corresponding assertion for the \textit{SFT} $X_{K,l}$. From Proposition \ref{p:mixing} we know that $h(X_{K,l})\le \log\,(l+1)$.

Fix $K,l$ and consider the state $\mathsf{e}=\{0,\dots ,l-1\}$ in \eqref{eq:e}. We are interested in the number of paths of length $K$ in $X_{K,l}$ which begin and end in $\mathsf{e}$. For this it will be convenient to work from right to left, starting from $\mathsf{e}$: by \eqref{eq:pa2}, $\mathsf{e}$ has $l+1$ predecessors $\mathsf{a}_{-1}^{(i)},\,i=1,\dots ,l+1$, say, each of which has a maximal element $\le l$. If $l<K-1$ we can repeat this argument and obtain, for each $\mathsf{a}_{-1}^{(i)}$, $l+1$ predecessors. This second generation of predecessors has maximal elements which are all $\le l+1$. After repeating this $K-l$ times we have found a total of $(l+1)^{K-l}$ distinct allowed paths of length $K-l+1$ in $X_{K,l}$, all of which have $\mathsf{e}$ as their final state. If $v$ is such a path with initial state $\mathsf{b}$, say, we can extend this path to the left by choosing $l$ successive predecessors of $\mathsf{b}$ until we arrive at $\mathsf{e}$ (as explained in the proof of Proposition \ref{p:mixing}).

This construction results in $(l+1)^{K-l}$ distinct allowed paths of length $K+1$ in $X_{K,l}$, all of which begin and end in $\mathsf{e}$. Since we can concatenate these paths arbitrarily (overlapping in the symbol $\mathsf{e}$), we have proved that $h(X_{K,l})\ge \frac1K \log\,((l+1)^{K-l}) = \frac{K-l}{K}\log\,(l+1)$, as claimed. The last assertion is a trivial consequence of this.
	\end{proof}

	\begin{rema}
	\label{r:symmetry}
According to Proposition \ref{p:comparisons} (3), $h(\Omega _{K,l})=h(\Omega _{K,K-l})$ for every $K\ge1$ and $l=0,\dots ,K$ (where $h(\Omega _{K,l})=0$ if $l=0$ or $l=K$). This allows us to symmetrise the first inequality in Theorem \ref{t:entropy} and to conclude that
	\begin{displaymath}
\tfrac1K\cdot \max\,\bigl((K-l)\cdot \log\,(l+1), l\cdot \log\,(K-l+1)\bigr) \le h(\Omega _{K,l}) \le \log\,(l+1)
	\end{displaymath}
for every $K\ge1$ and $l=0,\dots ,K$.
	\end{rema}

 For reasons of symmetry one would also expect that $h(\Omega _{K,l})$ is maximal if $l$ lies in the middle of the range $\{0,\dots ,K\}$, i.e., if $|\frac{K}{2}-l|\le \frac{1}{2}$. Our next proposition shows that this is indeed the case.

	\begin{prop}
	\label{p:maximal}
For every $K,l$ with $K\ge2$ and $1\le l\le \frac{K}{2}$,
	\begin{equation}
	\label{eq:entOmKl}
h(\Omega _{K,l-1}) \le h(\Omega _{K,l}).
	\end{equation}
	\end{prop}

For the proof of Proposition \ref{p:maximal} we need additional notation. For every finite set $\mathsf{u}\subset \mathbb{Z}$ and every $j\in \mathbb{Z}$ we set $\mathsf{u}+j=\{i+j:i\in \mathsf{u}\}$.

Fix $K,l$ with $K\ge 2$ and $0\le l\le K$. For any pair $\mathsf{u},\mathsf{v}$ of disjoint (and possibly empty) subsets of $\mathbb{Z}$ we set
	\begin{equation}
	\label{eq:AKluv}
\mathsf{B}_{K,l}^{[\mathsf{u},\mathsf{v}]}=
	\begin{cases}
\{\mathsf{a} \in \mathsf{B}_{K,l}:\mathsf{u}\subset \mathsf{a}\enspace \textup{and}\enspace \mathsf{a}\cap \mathsf{v}=\varnothing \} &\textup{if}\enspace (\mathsf{u}\cup\mathsf{v})\subset \mathsf{A}_{K-1},
	\\
\varnothing &\textup{otherwise},\vspace{-2mm}
	\end{cases}
	\end{equation}
and\vspace{-2mm}
	\begin{equation}
	\label{eq:XKluv}
X_{K,l}^{[\mathsf{u},\mathsf{v}]} = X_{K,l} \cap (\mathsf{B}_{K,l}^{[\mathsf{u},\mathsf{v}]})^\mathbb{Z}.
	\end{equation}
Both $\mathsf{B}_{K,l}^{[\mathsf{u},\mathsf{v}]}$ and $X_{K,l}^{[\mathsf{u},\mathsf{v}]}$ are empty whenever $|\mathsf{u}|>l$, and $\mathsf{B}_{K,l}^{[\varnothing ,\varnothing ]}=\mathsf{B}_{K,l}$ and $X_{K,l}^{[\varnothing ,\varnothing ]}=X_{K,l}$. If $|\mathsf{u}|=1$ with $\mathsf{u}=\{j\}$ for some $j\in \mathsf{A}_{K-1}$ we write $\mathsf{B}_{K,l}^{[j,\mathsf{v}]}$ and $X_{K,l}^{[j,\mathsf{v}]}$ instead of $\mathsf{B}_{K,l}^{[\{j\},\mathsf{v}]}$ and $X_{K,l}^{[\{j\},\mathsf{v}]}$. The case where $|\mathsf{v}|=1$ will be treated similarly.

\smallskip We set
	\begin{equation}
	\label{eq:Omega Kl 0K}
	\begin{gathered}
\Omega _{K,l}^{(0)} = \{\omega =(\omega _n)\in \Omega _{K,l}:\omega _n\ne0\enspace \textup{for every}\enspace n\in \mathbb{Z}\},
	\\
\Omega _{K,l}^{(K)} = \{\omega =(\omega _n)\in \Omega _{K,l}:\omega _n\ne K\enspace \textup{for every}\enspace n\in \mathbb{Z}\}.
	\end{gathered}
	\end{equation}
If $\phi =\phi _{K,l}\colon \Omega _{K,l}\longrightarrow X_{K,l}$ is the shift-equivariant isomorphism defined in \eqref{eq:phiKl}, then
	\begin{displaymath}
\phi (\Omega _{K,l}^{(0)}) = X_{K,l}^{[0,\varnothing ]}\enspace \enspace \textup{and}\enspace \enspace \phi (\Omega _{K,l}^{(K)}) = X_{K,l}^{[\varnothing ,K-1]}.
	\end{displaymath}
The reason for our interest in these subshifts is that
	\begin{displaymath}
\Omega _{K,l}^{(0)} \simeq \Omega _{K-1,l-1},\enspace \enspace \Omega _{K,l}^{(K)} \simeq \Omega _{K-1,l},
	\end{displaymath}
and hence
	\begin{displaymath}
X_{K,l}^{[0,\varnothing ]} \simeq X_{K-1,l-1}\enspace \enspace \textup{and}\enspace \enspace X_{K,l}^{[\varnothing ,K-1]}\simeq X_{K-1,l}.
	\end{displaymath}
This shows that \eqref{eq:entOmKl} is equivalent to the assertion that
	\begin{equation}
	\label{eq:entX0K}
h(X_{K+1,l}^{[0,\varnothing ]}) \le h(X_{K+1,l}^{[\varnothing ,K]})
	\end{equation}
for every $K\ge1$ and $l\le\frac{K+1}{2}$.

We set $Y=X_{K+1,l}^{[0,\varnothing ]}$, $Z=X_{K+1,l}^{[\varnothing ,K]}$, and write $Y_N$ and $Z_N$ for the set of all paths of length $N$ in $Y$ and $Z$, respectively. Since $h(Y)=\lim_{N\to\infty }\frac{1}{N}\log|Y_N|$ and $h(Z)=\lim_{N\to\infty }\frac{1}{N}\log|Z_N|$ we have to investigate the growth rates of the cardinalities $|Y_N|$ and $|Z_N|$ as $N\to\infty $.

We start with $Y_N$ and write, for all disjoint finite sets $\mathsf{u},\mathsf{v}\subset \mathbb{Z}$, $Y_N^{[\mathsf{u},\mathsf{v}]}\subset Y_N$ for the set of all paths of length $N$ ending in an element of $\mathsf{B}_{K+1,l}^{[\mathsf{u},\mathsf{v}]}$ (note that this set will be empty if $(\mathsf{u}\cup \mathsf{v})\not\subset \{1,\dots ,K\}$ or $|\mathsf{u}|\ge l$).

	\begin{lemm}
	\label{l:B0}
For every $N\ge1$ and every $\mathsf{u}\subset \{1,\dots ,K\}$,
	\begin{equation}
	\label{eq:2.6}
\bigl|Y_{N+1}^{[\mathsf{u},\varnothing ]}\bigr| = (K-l-|\mathsf{u}|+1)
\cdot \bigl|Y_N^{[\{1\}\cup (\mathsf{u}+1),\varnothing ]}\bigr| + \sum\nolimits_{j\in (\{1\}\cup (\mathsf{u}+1))} \bigl|Y_N^{[(\{1\}\cup (\mathsf{u}+1))\smallsetminus \{j\},\varnothing ]}\bigr|.
	\end{equation}
	\end{lemm}

	\begin{proof}
If $|\mathsf{u}|\ge l$, $\mathsf{B}^{[\{0\}\cup \mathsf{u},\varnothing ]}_{K+1,l}=\varnothing $ and hence $Y_N^{[\mathsf{u},\varnothing ]}=\varnothing $ for every $N\ge 0$. If $|\mathsf{u}|<l$, every $y\in Y_{N+1}^{[\mathsf{u},\varnothing ]}$ has the form $y=(\mathsf{a}_0,\mathsf{a}_1,\dots ,\mathsf{a}_{N-1},\mathsf{a}_{N})$ with $\mathsf{a}_i\in \mathsf{B}_{K,l}^{[0,\varnothing ]}$ for $i=0,\dots ,N$, $\mathsf{a}_{N}\supset \mathsf{u}$, and $\mathsf{a}_{N-1}\in \mathbf{p}(\mathsf{a}_{N})$. Since both $\mathsf{a}_{N-1}$ and $\mathsf{a}_N$ contain $0$, Equation \eqref{eq:pa1} implies that one of the following conditions is satisfied:
	\begin{enumerate}
	\item[(i)]\label{possibilities}
$\{1\}\cup (\mathsf{u}+1) \subset \mathsf{a}_{N-1}$, in which case $|\mathbf{f}(\mathsf{a}_{N-1})| = K-l+2$ and every successor of $\mathsf{a}_{N-1}$ (including $\mathsf{a}_N$, of course) is of the form $((\mathsf{a}_{N-1}\smallsetminus \{0\})-1)\cup\{j\}$ for some $j\in \{0,\dots ,K\}\smallsetminus (\mathsf{a}_{N-1}-1)$,
	\item[(ii)]
$\{1\}\cup (\mathsf{u}+1) \not\subset \mathsf{a}_{N-1}$, but $\mathsf{a}_{N-1}\supset (\{1\}\cup (\mathsf{u}+1))\smallsetminus \{j\}$ for some $j\in \{1\}\cup (\mathsf{u}+1)$. In this case $|\mathbf{f}(\mathsf{a}_{N-1})| = 1$ and $\mathbf{f}(\mathsf{a}_{N-1})=\{\mathsf{a}_N\}$.
	\end{enumerate}
If $K\notin \mathsf{u}$, we obtain that
	\begin{align*}
\bigl|Y_{N+1}^{[\mathsf{u},\varnothing ]}\bigr|&= (K-l+2)\cdot \bigl|Y_N^{[\{1\}\cup (\mathsf{u}+1),\varnothing ]}\bigr|  +\sum\nolimits_{j\in \{1\}\cup (\mathsf{u}+1)} \bigl|Y_N^{[(\{1\}\cup (\mathsf{u}+1))\smallsetminus \{j\},j]}\bigr|
	\\
&= (K-l-|\mathsf{u}|+1)\cdot \bigl|Y_N^{[\{1\}\cup (\mathsf{u}+1),\varnothing ]}\bigr| + \sum\nolimits_{j\in \{1\}\cup (\mathsf{u}+1)} \bigl|Y_N^{[(\{1\}\cup (\mathsf{u}+1))\smallsetminus \{j\},\varnothing ]}\bigr|,
	\end{align*}
where we have used that $\bigl|Y_N^{[\mathsf{w}\cup \{j\},\varnothing ]}\bigr| + \bigl|Y_N^{[\mathsf{w},j]}\bigr| = \bigl|Y_N^{[\mathsf{w},\varnothing ]}\bigr|$ whenever $\{j\}\cup \mathsf{w}\subset \{1,\dots ,K\}$ and $j\notin \mathsf{w}$.

If $K \in \mathsf{u}$, then
	\begin{align*}
\bigl|Y_{N+1}^{[\mathsf{u},\varnothing ]}\bigr|&= (K-l+2)\cdot \bigl|Y_N^{[\{1\}\cup (\mathsf{u}+1),\varnothing ]}\bigr|  +\sum\nolimits_{j\in \{1\}\cup (\mathsf{u}+1),\,j\le K} \;\bigl|Y_N^{[(\{1\}\cup (\mathsf{u}+1))\smallsetminus \{j\},j]}\bigr|
	\\
&\qquad \enspace  +\bigl|Y_N^{[(\{1\}\cup (\mathsf{u}+1))\smallsetminus \{K+1\},\varnothing ]}\bigr|=  \bigl|Y_N^{[(\{1\}\cup (\mathsf{u}+1))\smallsetminus \{K+1\},\varnothing ]}\bigr|,
	\end{align*}
since \eqref{eq:AKluv} -- \eqref{eq:XKluv} guarantee that all other expressions in the middle term of this equation vanish. In either case \eqref{eq:2.6} is satisfied, so that the lemma is proved.
	\end{proof}

In order to prove an analogous recursion formula for $Z=X_{K+1,l}^{[\varnothing ,K]}$ we denote by $Z_N^{[\mathsf{v},\varnothing ]}\subset Z_N$ the set of all paths of length $N$ which \textit{begin} with an element of $\mathsf{B}_{K+1,l}^{[\mathsf{v},K]}$ for some finite set $\mathsf{v}\subset \mathbb{Z}$. Note that $\mathsf{B}_{K+1,l}^{[\mathsf{v},K]}=\varnothing $ whenever $|\mathsf{v}|>l$ or $\mathsf{v}\not\subset \{0,\dots ,K-1\}$.

	\begin{lemm}
	\label{l:BK}
For every $N\ge1$ and every $\mathsf{v}\subset \{0,\dots ,K-1\}$ with $|\mathsf{v}|\le l$,
	\begin{equation}
	\begin{aligned}
	\label{eq:2.8}
\bigl|Z_{N+1}^{[\varnothing ,\mathsf{v}]}\bigr| &= (l-|\mathsf{v}|) \cdot \bigl|Z_N^{[\varnothing ,\{K-1\}\cup (\mathsf{v}-1)]}\bigr|
	\\
&\qquad + \sum\nolimits_{j\in (\{K-1\}\cup (\mathsf{v}-1))} \bigl|Z_N^{[\varnothing ,(\{K-1\}\cup (\mathsf{v}-1))\smallsetminus \{j\}]}\bigr|.
	\end{aligned}
	\end{equation}
If $0\in \mathsf{v}$, \eqref{eq:2.8} reduces to
	\begin{displaymath}
\bigl|Z_{N+1}^{[\varnothing ,\mathsf{v}]}\bigr| = \bigl|Z_N^{[\varnothing ,\{K-1\}\cup (\mathsf{v}'-1)]}\bigr|,
	\end{displaymath}
where $\mathsf{v}'=\mathsf{v}\smallsetminus \{0\}$.
	\end{lemm}

	\begin{proof}
For every $\mathsf{a}\in \mathsf{B}_{K,l}$ we set $\bar{\mathsf{a}}=\{K-a:a\in \mathsf{a}\}$. We recall the the definition of the isomorphism $\Psi_{K+1,l}\colon X_{K+1,l}\longrightarrow X_{K+1,K+1-l}$ in \eqref{eq:Psi} and note that $\Psi_{K+1,l}(Z_N^{[\varnothing ,\mathsf{v}]})=\overline{Y}_N^{[\bar{\mathsf{v}}, \varnothing]}$, where $\overline{Y}= X_{K+1,K-l+1}^{[0,\varnothing]}$. In particular, $| Z_M^{[\varnothing ,\mathsf{v}]} |=|\overline{Y}_M^{[\bar{\mathsf{v}},\varnothing]}|$ for every $M \ge 1$.

Assume for the moment that $0\notin \mathsf{v}$. From Lemma \ref{l:B0} it follows that
	\begin{align*}
\bigl|Z_{N+1}^{[\varnothing ,\mathsf{v}]}&\bigr|= \bigl|\overline{Y}_{N+1}^{[\bar{\mathsf{v}} ,\varnothing ]}\bigr| = (l-|\mathsf{v}|)\cdot \bigl|\overline{Y}_N^{[\{1\}\cup (\bar{\mathsf{v}}+1),\varnothing ]}\bigr|
	\\
&\qquad \qquad \qquad \qquad \qquad + \sum\nolimits_{j\in (\{1\}\cup (\bar{\mathsf{v}}+1))} \bigl|\overline{Y}_N^{[(\{1\}\cup (\bar{\mathsf{v}}+1))\smallsetminus \{j\},\varnothing ]}\bigr|
	\\
&= (l-|\mathsf{v}|)\cdot \bigl|Z_N^{[\varnothing, \overline{\{1\}\cup (\bar{\mathsf{v}}+1)} ]}\bigr| + \sum\nolimits_{j\in (\{1\}\cup (\bar{\mathsf{v}}+1))} \bigl|Z_N^{[\varnothing, \overline{(\{1\}\cup (\bar{\mathsf{v}}+1))\smallsetminus \{j\}} ]}\bigr|
	\\
&= (l-|\mathsf{v}|)\cdot \bigl|Z_N^{[\varnothing, \overline{\{1\}}\cup \overline{\bar{\mathsf{v}}+1} ]}\bigr| + \sum\nolimits_{j\in (\{1\}\cup (\bar{\mathsf{v}}+1))} \bigl|Z_N^{[\varnothing, (\overline{\{1\}}\cup \overline{\bar{\mathsf{v}}+1})\smallsetminus \overline{\{j\}} ]}\bigr|
	\\
&= (l-|\mathsf{v}|)
\cdot \bigl|Z_N^{[\varnothing ,\{K-1\}\cup (\mathsf{v}-1)]}\bigr|+ \sum\nolimits_{j\in (\{K-1\}\cup (\mathsf{v}-1))} \bigl|Z_N^{[\varnothing ,(\{K-1\}\cup (\mathsf{v}-1))\smallsetminus \{j\}]}\bigr|,
	\end{align*}
where we have used the facts that $\overline{\bar{\mathsf{v}}}=\mathsf{v},\;\overline{\bar{\mathsf{v}}+1}=K-(K-\mathsf{v}+1)=\mathsf{v}-1, \; \overline{\mathsf{u} \cup \mathsf{v}}=\bar{\mathsf{u}} \cup \bar{\mathsf{v}} ,\; \overline{\mathsf{u} \smallsetminus \mathsf{v}}=\bar{\mathsf{u}} \smallsetminus \bar{\mathsf{v}}$ and $|\mathsf{v}|=|\bar{\mathsf{v}}|$.

If $0\in \mathsf{v}$, then $K\in \bar{\mathsf{v}}$, and \eqref{eq:AKluv} -- \eqref{eq:XKluv} imply that
	\begin{align*}
\bigl|Z_{N+1}^{[\varnothing ,\mathsf{v} ]}\bigr|&= \bigl|\overline{Y}_{N+1}^{[\bar{\mathsf{v}} ,\varnothing ]}\bigr| = \bigl|\overline{Y}_N^{[(\{1\}\cup (\bar{\mathsf{v}}+1))\smallsetminus \{K+1\},\varnothing ]}\bigr| = \bigl|Z_N^{[\varnothing, \overline{(\{1\}\cup (\bar{\mathsf{v}}+1))\smallsetminus \{K+1\}} ]}\bigr|
	\\
&= \bigl|Z_N^{[\varnothing, \overline{\{1\}\cup ((\bar{\mathsf{v}}+1)\smallsetminus \{K+1\})} ]}\bigr| = \bigl|Z_N^{[\varnothing ,\{K-1\}\cup (\mathsf{v}'-1)]}\bigr|.
	\end{align*}
This proves \eqref{eq:2.8}.
	\end{proof}

Finally we investigate the relation between $Y_N^{[\cdot ,\varnothing]}$ and $Z_N^{[\varnothing, \cdot ]}$. We prove the following statement by induction on $N$.

	\begin{lemm}
	\label{l:relYZ}
For every $N\ge1, 1 \le l \le \frac{K}{2}, 0 \le m\le l$ and every $\mathsf{u}\subset \{1,\dots ,K\}$ with $|\mathsf{u}|=m$,
	\begin{equation}
	\label{eq:relYZ}
[K-l+1]_m \cdot \bigl|Y_{N}^{[\mathsf{u},\varnothing ]}\bigr| \le [l]_m \cdot \bigl|Z_N^{[\varnothing, \bar{\mathsf{u}}]}\bigr|
	\end{equation}
where
	\begin{displaymath}
\smash[t]{[x]_m=
	\begin{cases}x\cdot (x-1) \cdot \dots \cdot (x-m+1)&\textup{if}\enspace m\ge1,
	\\
1&\textup{if}\enspace m=0.
	\end{cases}}
	\end{displaymath}
	\end{lemm}

	\begin{proof}
Let $N=1$. Then $|Y_1^{[\mathsf{u},\varnothing ]}|= \binom{K-m}{l-m-1}$, since we are choosing $l-1-|\mathsf{u}|$ elements in the set $\{1,\dots ,K\}\smallsetminus \mathsf{u}$. Similarly we see that $|Z_1^{[\varnothing ,\bar{\mathsf{u}}]}| = \binom{K-m}{l}$. Then
	\begin{equation}
	\label{eq:hypothesis}
	\begin{aligned}
[K-l+1]_m \cdot \bigl|Y_{1}^{[\mathsf{u},\varnothing ]}\bigr| & =[K-l+1]_m \cdot \tbinom{K-m}{l-m-1} = \tfrac{(K-m)!}{(l-m-1)!(K-l-m+1)!}
	\\
&\le \tfrac{(K-m)!}{(l-m)!(K-l-m)!}= [l]_m \cdot   \tbinom{K-m}{l}=  [l]_m \cdot \bigl|Z_1^{[\varnothing, \bar{\mathsf{u}}]}\bigr|,
	\end{aligned}
	\end{equation}
where we have used the assumption $l \le K-l$. By using \eqref{eq:hypothesis} as our induction hypothesis, applying the Lemmas \ref{l:B0} and \ref{l:BK}, and remembering that $|\mathsf{u}|=|\bar{\mathsf{u}}|=m$, we get that
	\begin{align*}
[K-l&+1]_m \cdot \bigl|Y_{N+1}^{[\mathsf{u},\varnothing ]}\bigr|=[K-l+1]_{m+1} \cdot \bigl|Y_N^{[\{1\}\cup (\mathsf{u}+1),\varnothing ]}\bigr|
	\\
&\qquad \quad + \sum\nolimits_{j\in (\{1\}\cup (\mathsf{u}+1))} [K-l+1]_{m} \cdot \bigl|Y_N^{[(\{1\}\cup (\mathsf{u}+1))\smallsetminus \{j\},\varnothing ]}\bigr|
	\\
&\le [l]_{m+1}\cdot \bigl|Z_N^{[\varnothing, \overline{\{1\}\cup (\mathsf{u}+1)}]}\bigr| +  \sum\nolimits_{j\in \{1\}\cup (\mathsf{u}+1)} [l]_m \cdot \bigl|Z_N^{[\varnothing, \overline{(\{1\}\cup (\mathsf{u}+1))\smallsetminus \{j\}} ]}\bigr|
	\\
&= [l]_m \big\{ (l-|\bar{\mathsf{u}}|)\cdot \bigl|Z_N^{[\varnothing, \{K-1\}\cup (\bar{\mathsf{u}}-1) ]}\bigr|
	\\
&\qquad \quad+ \sum\nolimits_{j\in (\{K-1\}\cup (\bar{\mathsf{u}}-1))} \bigl|Z_N^{[\varnothing, \{K-1\}\cup (\bar{\mathsf{u}}-1)\smallsetminus \{j\} ]}\bigr| \big\}= [l]_m \cdot \bigl|Z_{N+1}^{[\varnothing, \bar{\mathsf{u}}]}\bigr|.\qedhere
	\end{align*}
	\end{proof}

	\begin{proof}[Proof of Proposition \ref{p:maximal}]
By taking $m=0$ (and hence $\mathsf{u}=\varnothing$) in Lemma \ref{l:relYZ} we obtain that $|Y_N|=\bigl|Y_{N}^{[\varnothing ,\varnothing ]}| \le \bigl|Z_{N}^{[\varnothing ,\varnothing ]}\bigr|=|Z_N|$ for every $N\ge 2$. As noted in the penultimate paragraph before Lemma \ref{l:B0} this guarantees that
\[ h(X_{K+1,l}^{[0,\varnothing ]})= h(Y)=\lim_{N\to\infty }\frac{1}{N}\log|Y_N| \le \lim_{N\to\infty }\frac{1}{N}\log|Z_N|=h(Z) = h(X_{K+1,l}^{[\varnothing ,K]}) \]
for every $K \ge 1$ and $l \le \frac{K+1}{2}$. We have proved \eqref{eq:entX0K} or, equivalently, \eqref{eq:entOmKl}.
	\end{proof}

\section{The parity cocycle}\label{s:parity}

If $\omega =(\omega _k)\in \Omega _K$ is a periodic point with period $p$, say, then
	\begin{displaymath}
\pi ^{(\omega )}_{(0,p)}\coloneqq\bigl(\tilde{\omega }_0\,(\textup{mod}\,p),\dots ,\tilde{\omega }_{0+p-1}\,(\textup{mod}\,p)\bigr)
	\end{displaymath}
is a permutation of $(0,\dots ,p-1)$ (cf. Lemma \ref{l:subshift} (3)). What is the parity (or sign) of this permutation? In this section we prove that these parities are determined by the function $a\colon \Omega _K\longrightarrow \mathbb{Z}$ in \eqref{eq:a} and a continuous cocycle $\mathsf{s}\colon \mathbb{Z}\times \Omega _K\longrightarrow \{\pm1\}$ for the shift $\sigma $ on $\Omega _K$.

	\begin{theo}
	\label{t:parity}
Let $K\ge1$, define $c\colon \Omega _K\longrightarrow \mathbb{Z}$ by
	\begin{equation}
	\label{eq:c1}
c(\omega ) = \bigl|\bigl\{k<0:\tilde{\omega }_k>\tilde{\omega }_0\bigr\}\bigr|= \bigl|\bigl\{k=-K+1,\dots ,-1:\tilde{\omega }_k>\tilde{\omega }_0\bigr\}\bigr|,
	\end{equation}
and let $\mathsf{c}\colon \mathbb{Z}\times \Omega _K\longrightarrow \mathbb{Z}$ be given by
	\begin{equation}
	\label{eq:c}
\mathsf{c}(n,\omega )=
	\begin{cases}
\sum_{k=0}^{n-1}c(\sigma ^k\omega )&\textup{if}\enspace n>0,
	\\
0&\textup{if}\enspace n=0,
	\\
-\mathsf{c}(-n,\sigma ^n\omega )&\textup{if}\enspace n<0.
	\end{cases}
	\end{equation}
Consider the multiplicative group $C_2\coloneqq\{\pm1\}\subset \mathbb{R}$ and define $s\colon \Omega _K\longmapsto C_2$ and $\mathsf{s}\colon \mathbb{Z}\times \Omega _K\longrightarrow C_2$ by setting $s(\omega )=(-1)^{c(\omega )+a(\omega )}$ and
	\begin{equation}
	\label{eq:s}
\smash{\mathsf{s}(n,\omega )=(-1)^{\mathsf{c}(n,\omega ) + na(\omega )} = \prod\nolimits_{k=0}^{n-1} s(\sigma ^k\omega ).}
	\end{equation}
Then $\mathsf{s}$ satisfies the cocycle equation
	\begin{equation}
	\label{eq:cocycle}
\mathsf{s}(m+n,\omega )=\mathsf{s}(m,\sigma ^n\omega )\mathsf{s}(n,\omega  )
	\end{equation}
for every $m,n\in\mathbb{Z}$ and $\omega \in \Omega _K$. Furthermore, if $\omega \in \Omega _{K}$ is periodic with period $p$, then the parity of the permutation $\pi ^{(\omega )}_{(0,p)}$ defined in Lemma \ref{l:subshift} \textup{(3)} is given by
	\begin{equation}
	\label{eq:sign}
\textup{sgn}\,\pi ^{(\omega )}_{(0,p)} = (-1)^{a(\omega )}\,\mathsf{s}(p,\omega ).
	\end{equation}
	\end{theo}

	\begin{proof}
The only statement requiring verification is \eqref{eq:sign}. We fix $\omega \in \Omega _K$, $p>K\ge 2$ and use the conventions of Figure~\ref{f:figure1}. We call a pair $(\mathbf{a}=(a_1,a_2),\mathbf{b}=(b_1,b_2))\in \mathbb{Z}^2$ an \textit{inversion} if $a_1<b_1$ and $a_2>b_2$ (i.e., if $\mathbf{b}$ lies above and to the right of $\mathbf{a}$ in Figure~\ref{f:figure1}). Let
	\begin{equation}
	\label{eq:inversions1}
\mathscr{I}(p) = \{(\mathbf{a},\mathbf{b})\in (\tilde{Q}\cup \tilde{A}')\times S(\omega ):(\mathbf{a},\mathbf{b})\enspace \textup{is an inversion}\}.
	\end{equation}
According to \eqref{eq:c}, $c(\omega )=|\{\mathbf{b}\in S(\omega ):((\omega _0,0),\mathbf{b})\;\textup{is an inversion}\}|$. Hence
	\begin{displaymath}
\mathsf{c}(p,\omega ) = |\mathscr{I}(p)|.
	\end{displaymath}

Now suppose that $\omega $ is periodic with period $p>K$. Then $S(\omega )$ is invariant under translation by $(p,p)$ (which in Figure~\ref{f:figure1} moves every point $p$ steps down and to the right), and $\tilde{A}'=\tilde{A}+(p,p)$.

We populate $A^*$ by translating all points in $\tilde{A}'$ into $A^*$ by adding $(-p,0)$ (or, equivalently, by adding $(0,p)$ to all points in $\tilde{A}$), leaving the points in $\tilde{Q}$ unchanged: for $\mathbf{a}\in \tilde{A} \cup \tilde{Q} \cup \tilde{A}'$ we put
	\begin{displaymath}
\mathbf{a}^*=
	\begin{cases}
\mathbf{a}&\textup{if}\enspace \mathbf{a}\in \tilde{Q},
	\\
\mathbf{a}-(p,0)&\textup{if}\enspace \mathbf{a}\in \tilde{A}',
	\\
\mathbf{a}+(0,p)&\textup{if}\enspace \mathbf{a}\in \tilde{A},
	\end{cases}
	\end{displaymath}
and we set $\tilde{A}^*=\{\mathbf{a}^*:\mathbf{a}\in \tilde{A}\}= \{\mathbf{a}^*:\mathbf{a}\in \tilde{A}'\}$. Under our assumption that $\omega $ has period $p$, the set $\tilde{S}\coloneqq \tilde{Q}\cup \tilde{A}^*$ has exactly one element in each row and column of the square $Q$. Hence $\tilde{S}$ determines a permutation of $\{0,\dots ,p-1\}$ which coincides with $\pi ^{(\omega )}_{(0,p)}$. The parity of this permutation is the parity of the number of inversions occurring in $\pi ^{(\omega )}_{(0,p)}$: if
	\begin{equation}
	\label{eq:inversions}
\mathscr{J}(p) = \{(\mathbf{a},\mathbf{b})\in \tilde{S}\times \tilde{S}:(\mathbf{a},\mathbf{b})\enspace \textup{is an inversion}\},
	\end{equation}
then
	\begin{equation}
	\label{eq:sign pi}
\textup{sgn}\,\pi ^{(\omega )}_{(0,p)}=(-1)^{|\mathscr{J}(p)|}.
	\end{equation}
When comparing this with
	\begin{equation}
	\label{eq:sign s}
\mathsf{s}(p,\omega )=(-1)^{|\mathscr{I}(p)|+pa(\omega )},
	\end{equation}
we have to consider several cases:
	\begin{enumerate}
	\item[(a)]
If $(\mathbf{a},\mathbf{b})\in \tilde{Q}$ or $(\mathbf{a},\mathbf{b})\in \tilde{A}'$, then $(\mathbf{a},\mathbf{b})\in \mathscr{I}(p)$ if and only if $(\mathbf{a}^*,\mathbf{b}^*)\in \mathscr{J}(p)$.
	\item[(b)]
If $\mathbf{a}=(a_1,a_2)\in \tilde{B}$ and $\mathbf{b}=(b_1,b_2)\in \tilde{A}$, then $(\mathbf{a},\mathbf{b})\in\mathscr{I}(p)$ if and only if $a_1<b_1$, in which case $(\mathbf{b}^*,\mathbf{a}^*)\notin \mathscr{J}(p)$. If $a_1>b_1$, then $(\mathbf{a},\mathbf{b})\notin \mathscr{I}(p)$, but $(\mathbf{b}^*,\mathbf{a}^*)\in \mathscr{J}(p)$. Since $|\tilde{B}\times \tilde{A}|=a(\omega )(K-a(\omega ))$, we conclude that $|\mathscr{I}(p) \cap (\tilde{B}\times \tilde{A})| + |\mathscr{J}(p) \cap (\tilde{A}^*\times \tilde{B})|= a(\omega )(K-a(\omega ))$.
	\item[(c)]
$(\tilde{A}'\times \tilde{C})\cap \mathscr{I}(p)= (\tilde{C}\times \tilde{A}')\cap \mathscr{I}(p)=\varnothing $, but $(\tilde{A}^*\times \tilde{C})\subset \mathscr{J}(p)$. Note that $|\tilde{A}'\times \tilde{C}|= |\tilde{C}\times \tilde{A}'| = |\tilde{A}^*\times \tilde{C}|=a(\omega )(p-2K+a(\omega ))$.
	\item[(d)]
If $\mathbf{a}=(a_1,a_2)\in \tilde{D}$ and $\mathbf{b}=(b_1,b_2)\in \tilde{A}'$, then $(\mathbf{a},\mathbf{b})\in \mathscr{I}(p)$ if and only if $a_2<b_2$, in which case $(\mathbf{b}^*,\mathbf{a}^*)\notin \mathscr{J}(p)$. If $a_2>b_2$, then $(\mathbf{a},\mathbf{b})\notin \mathscr{I}(p)$, but $(\mathbf{b}^*,\mathbf{a}^*)\in \mathscr{J}(p)$. Since $|\tilde{A}'\times  \tilde{D}|=a(\omega )(K-a(\omega ))$, we conclude that $|\mathscr{I}(p) \cap (\tilde{D}\times \tilde{A}')| + |\mathscr{J}(p) \cap (\tilde{A}^*\times \tilde{D})|=a(\omega )(K-a(\omega ))$.
	\end{enumerate}
Clearly,
	\begin{align*}
|\mathscr{I}(p)|&=|\mathscr{I}(p)\cap (\tilde{Q}\times \tilde{Q})| + |\mathscr{I}(p)\cap (\tilde{A}'\times \tilde{A}')| + |\mathscr{I}(p)\cap (\tilde{A}'\times \tilde{Q})|
	\\
& \qquad \qquad + |\mathscr{I}(p)\cap (\tilde{Q}\times \tilde{A})| + |\mathscr{I}(p) \cap (\tilde{Q}\times \tilde{A}')|
	\\
&= |\mathscr{I}(p)\cap (\tilde{Q}\times \tilde{Q})| + |\mathscr{I}(p)\cap (\tilde{A}'\times \tilde{A}')|
	\\
&\qquad \qquad + |\mathscr{I}(p)\cap (\tilde{B}\times \tilde{A})| + |\mathscr{I}(p)\cap (\tilde{D}\times \tilde{A}')|,
	\\
|\mathscr{J}(p)|&=|\mathscr{J}(p)\cap (\tilde{Q}\times \tilde{Q})| + |\mathscr{J}(p)\cap (\tilde{A}^*\times \tilde{A}^*)|
	\\
&\qquad \qquad + |\mathscr{J}(p)\cap (\tilde{A}^*\times \tilde{Q})| + |\mathscr{J}(p)\cap (\tilde{Q}\times \tilde{A}^*)|
	\\
&=|\mathscr{J}(p)\cap (\tilde{Q}\times \tilde{Q})| + |\mathscr{J}(p)\cap (\tilde{A}^*\times \tilde{A}^*)| + |\mathscr{J}(p)\cap (\tilde{A}^*\times \tilde{B})|
	\\
&\qquad \qquad + |\mathscr{J}(p)\cap (\tilde{A}^*\times \tilde{D})| + |\tilde{A}^*\times \tilde{C}|
	\end{align*}
By combining the cases (a) -- (d) listed above and remembering that $|\tilde{C}|=p-2K+a(\omega )$ we obtain that
	\begin{align*}
\mathscr{J}(p)&=|\mathscr{I}(p)\cap (\tilde{Q}\times \tilde{Q})| + |\mathscr{I}(p)\cap (\tilde{A}'\times \tilde{A}')| - |\mathscr{I}(p)\cap (\tilde{B}\times \tilde{A})|
	\\
&\qquad \qquad - |\mathscr{I}(p)\cap (\tilde{D}\times \tilde{A}')| + 2a(\omega )(K - a(\omega )) + a(\omega )(p-2K+a(\omega )).
	\end{align*}
Hence
	\begin{displaymath}
|\mathscr{J}(p)|=|\mathscr{I}(p)| + a(\omega )(p+a(\omega )) = |\mathscr{I}(p)| + a(\omega )p+a(\omega )\pmod2.
	\end{displaymath}
If we recall \eqref{eq:sign pi}--\eqref{eq:sign s} we obtain \eqref{eq:sign}.
	\end{proof}

	\begin{exas}
	\label{e:1}
(1) Let $K=2$ (cf. Example \ref{e:sft}). The map $c\colon \Omega _2\longrightarrow \mathbb{Z}$ in \eqref{eq:c1} is given by
	\begin{displaymath}
\smash[b]{c(\omega )=
	\begin{cases}
1&\textup{if}\enspace \omega _{0}=0\;\textup{and}\;\omega _{-1}=2,
	\\
0&\textup{otherwise},
	\end{cases}}
	\end{displaymath}
and
	\begin{displaymath}
\mathsf{c}(n,\omega )=|\{k:0\le k<n:\omega _k=0\enspace \textup{and}\enspace \omega _{k-1}=2\}|.
	\end{displaymath}

\smallskip (2) More generally, if $K\ge1$ and $\omega \in \Omega _{K,1}$, then
	\begin{displaymath}
c(\omega )=
	\begin{cases}
1&\textup{if}\enspace \omega _0=0,
	\\
0&\textup{otherwise},
	\end{cases}
	\end{displaymath}
and $\mathsf{c}(n,\omega )=|\{k:0\le k<n:\omega _k=0\}|$.

\smallskip (3) For $\omega \in \Omega _{K,l}$ with $1\le l \le K-1$, $c(\omega )$ can be calculated by using the isomorphism $\phi _{K,l}\colon \Omega _{K,l}\longrightarrow X_{K,l}$ in \eqref{eq:phiKl}: $c(\omega ) = |\{a\in \phi _{K,l}(\omega )_0: a>\omega _0\}|$, where $\omega _0$ is determined by \eqref{eq:iso1}.
	\end{exas}

As one would expect, the parity cocycle $\mathsf{s}\colon \mathbb{Z}\times \Omega _K\longrightarrow C_2$ in \eqref{eq:s} is \textit{nontrivial} in the sense that its group of essential values is equal to $C_2$. More precisely, the following is true.

	\begin{prop}
	\label{p:ergodic}
Let $K\ge2$, $1\le l \le K-1$, and let $\mu _{K,l}$ be the unique shift-invariant probability measure with maximal entropy on $\Omega _{K,l}$. Then the skew-product transformation $\tilde{\sigma }\colon \Omega _{K,l}\times C_2\longrightarrow \Omega _{K,l}\times C_2$, defined by
	\begin{displaymath}
\tilde{\sigma }(\omega ,j) = (\sigma \omega ,s(\omega )j)
	\end{displaymath}
for every $(\omega ,j)\in \Omega _{K,l}\times C_2$, is ergodic with respect to the product measure $\tilde{\mu }_{K,l}=\mu _{K,l}\times \lambda _{C_2}$, where $\lambda _{C_2}(\{1\})=\lambda _{C_2}(\{-1\})=\frac12$.
	\end{prop}

	\begin{proof}
According to \cite[Corollary 5.4]{CocycleNotes}, we have to show the following: for every Borel set $B\subset \Omega _{K,l}$ with $\mu _{K,l}(B)>0$ and every $j\in C_2$,
	\begin{equation}
	\label{eq:essential}
B\cap \sigma ^{-m}B \cap \{\omega \in \Omega _{K,l}:\mathsf{s}(m,\omega )=j\}\ne \varnothing
	\end{equation}
for some $m>0$.

Consider the cylinder sets
	\begin{gather*}
E=\{\omega \in \Omega _{K,l}: \omega _i=l\enspace \textup{for}\enspace i=0,\dots ,K+1\},
	\\
F=\{\omega \in \Omega _{K,l}: \omega _i=l\; \textup{for}\; i=0,\dots ,K-2,K+1, \; \omega _{K-1}=l+1, \; \omega _K=l-1\}.
	\end{gather*}
Since $\Omega _{K,l}$ is irreducible and aperiodic, $\beta \coloneqq \mu _{K,l}(D)>0$, where $D=E\cup F$. We define a homeomorphism $V\colon \Omega _{K,l}\longrightarrow \Omega _{K,l}$ by setting $V\omega =\omega $ if $\omega \notin D$, and
	\begin{displaymath}
(V\omega )_i=
	\begin{cases}
l&\textup{if}\enspace \omega \in E\enspace \textup{and}\enspace i=0,\dots ,K-2,K+1,
	\\
l+1&\textup{if}\enspace \omega \in E\enspace \textup{and}\enspace i=K-1,
	\\
l-1&\textup{if}\enspace \omega \in E\enspace \textup{and}\enspace i=K,
	\\
l&\textup{if}\enspace \omega \in F\enspace \textup{and}\enspace i=0,\dots ,K+1.
	\end{cases}
	\end{displaymath}
Then $V^2=\textup{Id}_{\Omega _{K,l}}$ and $VE=F$.

We fix $B\subset \Omega _{K,l}$ with $\mu _{K,l}(B)>0$. By approximating $B$ with closed and open subsets of $\Omega _{K,l}$ we see that $\lim_{m\to\infty }\mu _{K,l}(B\cap \sigma ^{-m}V\sigma ^m B) = \mu _{K,l}(B)$ and $\lim_{m\to\infty }\mu _{K,l}(\sigma ^{-2m}B \cap \sigma ^{-m}V\sigma ^{-m} B)= \mu _{K,l}(B)$. Furthermore, since $\mu _{K,l}$ is mixing of every order, $\lim_{m\to\infty }\mu _{K,l}(B\linebreak[0]\cap \sigma ^{-m}D \cap \sigma ^{-2m}B)= \beta \mu _{K,l}(B)^2$. We conclude that
	\begin{align*}
\beta \mu _{K,l}(B)^2&= \lim_{m\to\infty }\mu _{K,l}(B\cap \sigma ^{-m} D \cap \sigma ^{-2m}B)
	\\
&= \lim_{m\to\infty }\mu _{K,l}(B\cap \sigma ^{-m}V\sigma ^m B \cap \sigma ^{-m} D \cap \sigma ^{-2m}B \cap \sigma ^{-m}V\sigma ^{-m}B).
	\end{align*}

Let $m>K$ be sufficiently large so that the set $E=B\cap \sigma ^{-m}V\sigma ^m B \cap \sigma ^{-m} D \cap \sigma ^{-2m}B \linebreak[0]\cap \sigma ^{-m}V\sigma ^{-m}B$ is nonempty. For every $\omega \in E$ the following conditions are satisfied:
	\begin{align*}
\omega &\in B, & \sigma ^{2m}\omega &\in B, & \sigma ^m\omega &\in D,
	\\
\omega '&\in B, & \sigma ^{2m} \omega '&\in B, & \sigma ^m\omega '&\in D,
	\end{align*}
where $\omega '=\sigma ^{-m}V\sigma ^m\omega $. A glance at the definition of the cocycle $\mathsf{s}$ in \eqref{eq:s} shows that, for every $\omega \in E$ and $k\in \mathbb{Z}$,
	\begin{displaymath}
s(\sigma ^k\omega ')=
	\begin{cases}
\hphantom{-}s(\sigma ^k\omega )&\textup{if}\enspace k\ne m+K,
	\\
-s(\sigma ^k\omega )&\textup{if}\enspace k=m+K.
	\end{cases}
	\end{displaymath}
In particular, $\{\omega ,\omega '\}\subset B\cap \sigma ^{-2m}B$ and $\mathsf{s}(2m, \omega )=-\mathsf{s}(2m,\omega ')$. This proves \eqref{eq:essential}.
	\end{proof}

	\begin{exam}[Generalized circulants]
	\label{e:circulant}
The study of the space $\Omega _K$, its irreducible components $\Omega _{K,l}$, and the permutations corresponding to periodic elements of $\Omega _K$ was partly motivated by expressions occurring in the calculation of entropy of (expansive) algebraic actions of the discrete Heisenberg group (cf. \cite[Section 8]{Lind+Schmidt}).

Let $\phi _0,\dots ,\phi _K$ be continuous complex-valued functions on $\mathbb{T}$. We are interested in bi-infinite `generalized circulants' of the form
	\begin{equation}
	\label{eq:circulant}
A_{t,\alpha }=\scalebox{.90}{$\left(
	\begin{smallmatrix}
&\vdots\enspace  & \vdots & \vdots & \vdots & \vdots & \vdots & \vdots &\vdots &\enspace \vdots &
	\\
\cdots\enspace  & 0\enspace  & \phi _{0}(t-\alpha ) & \phi _{1}(t-\alpha ) & \cdots & \phi _{K-1}(t-\alpha ) & \phi _{K}(t-\alpha ) & 0&0&\enspace 0&\enspace \cdots
	\\
\cdots\enspace  & 0\enspace  & 0 &\phi _{0}(t) & \phi _{1}(t) & \cdots & \phi _{K-1}(t) & \phi _{K}(t) & 0&\enspace 0&\enspace \cdots
	\\
\cdots\enspace  & 0\enspace  & 0 &0& \phi _{0}(t+\alpha ) & \phi _{1}(t+\alpha ) & \cdots & \phi _{K-1}(t+\alpha )  &\phi _K(t + \alpha )& \enspace 0&\enspace \cdots\vspace{-2mm}
	\\
 & \vdots\enspace & \vdots & \vdots & \vdots & \vdots & \vdots & \vdots &\vdots &\enspace \vdots &
	\end{smallmatrix}
\right) $}
	\end{equation}
for $(t,\alpha )\in \mathbb{T}^2$. The matrix $A_{t,\alpha }$ acts by left multiplication on the space $\ell ^\infty (\mathbb{Z},\mathbb{C})^\top$ of bounded column vectors with complex entries. If $\alpha $ is rational with $\alpha =p/q$ in lowest terms, say, then $A_{t,p/q}$ acts by left multiplication on the set of elements in $\ell ^\infty (\mathbb{Z},\mathbb{C})^\top$ with period $q$, which we identify with $\mathbb{C}^q \simeq \ell ^\infty (\mathbb{Z}/q\mathbb{Z},\mathbb{C})$. The determinant of this linear transformation of $\mathbb{C}^q$ can be expressed in terms of \textit{SFT} $\Omega _K$, using the parity cocycle \eqref{eq:s} -- \eqref{eq:cocycle}: if $P_q(\Omega _K)$ and $P_q(\Omega _{K,l})$ are the sets of points of period $q$ in $\Omega _K$ and $\Omega _{K,l}$, then
	\begin{equation}
	\label{eq:determinant1}
	\begin{aligned}
\det A_{t,p/q} & = \sum\nolimits_{\omega \in P_q(\Omega _K)} \;(-1)^{a(\omega )} \,\mathsf{s}(q,\omega )\prod\nolimits_{j=0}^{q-1} \phi _{\omega _j}(t+jp/q)
	\\
& = \sum\nolimits_{l=0}^K \; (-1)^l \cdot \sum\nolimits_{\omega \in P_q(\Omega _{K,l})}\;\prod\nolimits_{j=0}^{q-1} s(\sigma ^j\omega )\phi _{\omega _j}(t+jp/q).
	\end{aligned}
	\end{equation}
As explained in \cite[Section 8]{Lind+Schmidt} one should normalize $\det A_{t,p/q}$ by setting
	\begin{equation}
	\label{eq:determinant2}
D(A_{t,p/q}) = |\det (A_{t,p/q})|^{1/q}.
	\end{equation}
In the context of algebraic actions of the discrete Heisenberg group the functions $\phi _i,\,i=0,\dots ,K$, are trigonometric polynomials arising from the element $f$ in the integer group ring $\mathbb{Z}\Gamma $ which defines the action, and the quantity $\int_\mathbb{T} \log D(A_{t,p/q}) dt$ measures the contribution to the entropy of this action associated with a \textit{rational} rotation number $\alpha =p/q$ representing the central generator $z=$\scalebox{.7}{$\Bigl( \begin{smallmatrix}1&0&1\\0&1&0\\0&0&1 \end{smallmatrix} \Bigr)$} of $\Gamma $. If this action is expansive, the asymptotic behaviour (as $q\to\infty $) of the expressions $D(A_{t,p/q})$ in \eqref{eq:determinant2} determines the entropy of the algebraic action (cf. \cite{DS} and \cite[Section 8]{Lind+Schmidt}). For nonexpansive actions of this form one might still expect that $\limsup_{q\to\infty }D(A_{t,p/q})=h(\alpha _f)$ for every $t\in \mathbb{T}$, but this statement is currently only conjectural.
	\end{exam}

\section{Permutations of $\mathbb{Z}^d$ with restricted movement}\label{s:Z2}

We start by still assuming that $d=1$, but by allowing $\mathsf{A}\subset \mathbb{Z}$ to be an arbitrary nonempty finite set. Define $\Pi _\mathsf{A}$ and $\Omega _\mathsf{A}$ as at the beginning of Section \ref{s:Intro}. Then $\Omega _\mathsf{A}$ is a \textit{SFT} by Lemma \ref{l:subshift}.

	\begin{prop}
	\label{p:OmegaA}
\textup{(1)} If $|\mathsf{A}|\ge 2$, then $\Omega _\mathsf{A}$ is not irreducible and hence not mixing.

\textup{(2)} The \textit{SFT} $\Omega _\mathsf{A}$ is finite if $|\mathsf{A}|\le 2$, and has positive topological entropy if $|\mathsf{A}|\ge3$.
	\end{prop}

	\begin{proof}
In view of Proposition \ref{p:PiA} we assume that $0\in \mathsf{A}\subset \mathbb{Z}_+=\{0,1,2,\dots \}$ and choose $K\ge 1$ so that $\mathsf{A}\subset \mathsf{A}_K$ (cf. \eqref{eq:PiK}). For the proof of (1) we note that, for every $a\in \textsf{A}$, the fixed point $\underline{a} = (\dots , a,a,a,\dots )$ lies in $\Omega _\textsf{A} \cap \Omega _{K,a}$. For $a,b\in \mathsf{A}$ with $a<b$, the fixed points $\underline{a}$ and $\underline{b}$ lie in the distinct irreducible components $\Omega _{K,a}$ and $\Omega _{K,b}$ of $\Omega _K$. The sets $\Omega _\mathsf{A} \cap \Omega _{K,a}$ and $\Omega _\mathsf{A} \cap \Omega _{K,b}$ are disjoint, nonempty, shift-invariant, open subsets of $\Omega _\mathsf{A}$, so that $\Omega _\mathsf{A}$ cannot be topologically mixing. This proves (1).

We turn to (2). If $\mathsf{A}=\{0\}$ then $\Omega _\mathsf{A}=\{\underline{0}\}$. If $\mathsf{A}=\{0,a\}$ for some $a\ge1$, then $|\Omega _\mathsf{A}|=|\Pi _\mathsf{A}|=2^a$: for every $u=(u_0,\dots ,u_{a-1})\in \{0,1\}^a$ there exists a unique $\pi ^{(u)}\in \Pi _\mathsf{A}$ such that $\pi ^{(u)}(i+ma)=i+(m+u_i)a$ for every $i=0,\dots ,a-1$ and $m\in \mathbb{Z}$. Furthermore, every $\pi \in \Pi _\mathsf{A}$ is of this form for some $u\in \{0,1\}^a$.

Finally we assume that $\mathsf{A}\supset \{0,a,b\}$ with $0<a<b$. For every $n\in \mathbb{Z}$ we consider the finite permutation of $\mathbb{Z}$ defined by
	\begin{displaymath}
	\tau _n= \Bigl(\begin{smallmatrix}
n&\;n+1&\;\cdots&\;n+a-1&\;n+a&\;n+a+1&\;\cdots &\;n+b-1
	\\
n+b-a&\;n+b-a+1&\;\cdots&\;n+b-1&\;n&\;n+1&\;\cdots &\; n+b-a-1
	\end{smallmatrix}\Bigr),
	\end{displaymath}
Clearly, the permutations $\tau _m$ and $\tau _n$ commute if $|m-n|\ge b$. This allows us to define, for every $v=(v_n)\in\{0,1\}^\mathbb{Z}$, a permutation $\tau ^{(v)}\in S^\infty (\mathbb{Z})$ by setting
	\begin{displaymath}
\tau ^{(v)} = \prod\nolimits_{n\in \mathbb{Z}} \tau _{bn}^{v_n},
	\end{displaymath}
where $\tau _m^0$ is the identity permutation for every $m\in \mathbb{Z}$. Note that $\varsigma ^{a}\circ \tau ^{(v)} \in \Pi _\mathsf{A}$ for every $v\in\{0,1\}^\mathbb{Z}$ (cf. Proposition \ref{p:PiA}).

In order to understand what is going on here it may help to consider an element $v\in\{0,1\}^\mathbb{Z}$ of the form $v= (\dots ,1,\dot 0,1, \dots)$, where the dot marks the zero-th coordinate of $v$. Then $\tau ^{(v)}$ is the permutation
	\begin{displaymath}
\scalebox{.9}{$\Bigl(\begin{smallmatrix}
\cdots &\, \vert &\, -b&\,\cdots&\,-b+a-1&\,-b+a&\,\cdots &\,-1 &\, \vert &\, 0&\,\cdots&\,a-1&\,a&\,\cdots &\,b-1 &\,\vert &\, b&\,\cdots&\,b+a-1&\,b+a&\,\cdots &\,2b-1 &\, \vert &\, \cdots
	\\
\cdots &\, \vert &\, -a&\,\cdots&\,-1&\,-b&\,\cdots &\,-a-1 &\, \vert &\, 0&\,\cdots&\,a-1&\,a&\,\cdots &\,b-1 &\,\vert &\, 2b-a&\,\cdots&\,2b-1&\,b&\,\cdots &\,2b-a-1 &\, \vert &\, \cdots
	\end{smallmatrix} \Bigr),$}
	\end{displaymath}
where we have separated the individual permutations $\dots ,\tau _{-b}, \tau _0, \tau _b^0,\dots $ by vertical bars. The permutation $\varsigma ^{a}\circ \tau ^{(v)}$ is of the form
	\begin{displaymath}
\scalebox{.9}{$\Bigl(\begin{smallmatrix}
\cdots &\, \vert &\, -b&\,\cdots&\,-b+a-1&\,-b+a&\,\cdots &\,-1 &\, \vert &\, 0&\,\cdots&\,a-1&\,a&\,\cdots &\,b-1 &\,\vert &\, b&\,\cdots&\,b+a-1&\,b+a&\,\cdots &\,2b-1 &\, \vert &\, \cdots
	\\
\cdots &\, \vert &\, 0&\,\cdots&\,a-1&\,-b+a &\,\cdots &\,-1 &\, \vert &\, a&\,\cdots&\,2a-1&\,2a&\,\cdots &\,a+b-1 &\,\vert &\, 2b&\,\cdots&\,a+2b-1&\,a+b&\,\cdots &\,2b-1 &\, \vert &\, \cdots
	\end{smallmatrix} \Bigr)$}
	\end{displaymath}
and obviously lies in $\Pi _\mathsf{A}$. By doing this for every $v\in \{0,1\}^\mathbb{Z}$ we have --- in effect --- embedded a full two-shift in the coordinates $b\mathbb{Z}$ of $\Omega _\mathsf{A}$. This implies that $\Omega _\mathsf{A}$ has entropy $\ge \frac 1b \log2$.
	\end{proof}

Next we take a look at the case where $d\ge2$ and consider dynamical properties of the $\mathbb{Z}^d$-\textit{SFT} $\Omega _\mathsf{A}\subset \mathsf{A}^{\mathbb{Z}^d}$, such as entropy and topological mixing.

	\begin{theo}
	\label{t:posent}
If $\mathsf{A}$ is a finite subset of $\mathbb{Z}^2$, then $\Omega _\mathsf{A}$ has positive entropy if and only if $|\mathsf{A}|\ge 3$.
	\end{theo}

	\begin{proof}
We assume without loss in generality that $\mathbf{0}=(0,0) \in \mathsf{A}$: otherwise we replace $\mathsf{A}$ by $\mathsf{A}'=\mathsf{A}-\mathbf{u}$ for some $\mathbf{u}\in \mathsf{A}$. Then $\mathbf{0}\in \mathsf{A}'$, and $\Omega_{\mathsf{A}'}$ is topologically conjugate to $\Omega _\mathsf{A}$ with conjugating map $\phi_\mathsf{u}\colon \Omega _{\mathsf{A}'}\longrightarrow \Omega _\mathsf{A}$ given by $\phi _\mathbf{m}(\omega )_\mathbf{n} =\omega _\mathbf{n}+\mathsf{u}$ for every $\omega \in \Omega_\mathsf{A}$ and $\mathbf{n}\in \mathbb{Z}^2$.

Suppose that $\mathsf{A}$ is not contained in a one-dimensional subspace of $\mathbb{R}^2$. Then we can find two elements $\mathbf{a},\mathbf{b}\in \mathsf{A}$ which are linearly independent over $\mathbb{R}$. We denote by $\Delta \subset \Gamma \subset \mathbb{Z}^2$ the subgroups generated by $\{\mathbf{a}+\mathbf{b},3\mathbf{a}\}$ and $\{\mathbf{a},\mathbf{b}\}$, respectively. For every $x=(x_\mathbf{m})_{\mathbf{m}\in \Delta }\in \{0,1\}^\Delta $ we define a permutation $\pi _x\in \Pi _\mathsf{A}$ as follows:
	\begin{enumerate}
	\item[(i)]
if $\mathbf{n}\in \mathbb{Z}^2\smallsetminus \Gamma $ we set $\pi _x(\mathbf{n})=\mathbf{n}$;
	\item[(ii)]
if $\mathbf{n}\in \Delta $ and $x_\mathbf{n}=0$ we set $\pi _x(\mathbf{n})=\mathbf{n}+\mathbf{a}$, $\pi _x(\mathbf{n} + \mathbf{a}) = \mathbf{n}+\mathbf{a}+\mathbf{b}$, $\pi _x(\mathbf{n}+2\mathbf{a}) = \mathbf{n}+2\mathbf{a}$, and $\pi _x(\mathbf{n}+\mathbf{b}) = \mathbf{n}+\mathbf{b}$;
	\item[(iii)]
if $\mathbf{n}\in \Delta $ and $x_\mathbf{n}=1$ we set $\pi _x(\mathbf{n})=\mathbf{n}+\mathbf{b}$, $\pi _x(\mathbf{n} + \mathbf{b}) = \mathbf{n}+\mathbf{a}+\mathbf{b}$, $\pi _x(\mathbf{n}+\mathbf{a}) = \mathbf{n}+\mathbf{a}$, and $\pi _x(\mathbf{n}+2\mathbf{a}) = \mathbf{n}+2\mathbf{a}$.
	\end{enumerate}
It is easy to check that $\pi _x$ is a permutation of $\mathbb{Z}^2$ which lies in $\Pi _\mathsf{A}$. By varying $x$ in $\{0,1\}^\Delta $ we conclude that $h(\Omega _\mathsf{A})\ge \log 2/|\mathbb{Z}^2/\Delta |$.

If $\mathsf{A}$ is contained in a one-dimensional subspace of $\mathbb{R}^2$, we can find a primitive element $\mathbf{v}\in \mathbb{Z}^2$ such that $\mathsf{A} \subset S \coloneqq \mathbb{Z}\mathbf{v}=\{k\mathbf{v}:k\in \mathbb{Z}\} \cong \mathbb{Z}$. Choose a second primitive element $\mathbf{w}\in \mathbb{Z}^2$ so that $\{\mathbf{v},\mathbf{w}\}$ forms a basis of $\mathbb{Z}^2$. The restriction (or projection) $\Omega _\mathsf{A}|_S$ of $\Omega _\mathsf{A}$ to $S=\mathbb{Z}\mathbf{v}=\{k\mathbf{v}:k\in \mathbb{Z}\} \cong \mathbb{Z}$ is a \textit{SFT} under the shift $\sigma ^\mathbf{v}$, which has positive entropy if and only if $|\mathsf{A}| = |\mathsf{A}\cap S|\ge 3$ (Proposition \ref{p:OmegaA}). Since the restrictions of $\Omega _\mathsf{A}$ to $S+k\mathbf{w},\,k\in \mathbb{Z}$, are all isomorphic and independent of each other, they all have the same entropy $h(\Omega _\mathsf{A}|_S)$ under $\sigma ^\mathbf{v}$, and $h(\Omega _\mathsf{A}) = h(\Omega _\mathsf{A}|_S) >0$ if and only if $|\mathsf{A}|\ge3$. This completes the proof of Theorem \ref{t:posent}.
	\end{proof}

The proof of Theorem \ref{t:posent} yields two corollaries. For the proof of the first of these corollaries we recall that a nonzero subgroup $\Gamma \subset \mathbb{Z}^d$ is called \textit{primitive} if $\mathbb{Z}^d/\Gamma $ is torsion-free. Similarly, a nonzero element $\mathbf{n}\in \mathbb{Z}^d$ is \textit{primitive} if the cyclic subgroup $\{k\mathbf{n}:k\in \mathbb{Z}\}\subset \mathbb{Z}^d$ is primitive.

	\begin{coro}
	\label{c:posent}
If $\mathsf{A}$ is a finite subset of $\mathbb{Z}^d$, $d\ge2$, then $\Omega _\mathsf{A}$ has positive entropy if and only if $|\mathsf{A}|\ge 3$.
	\end{coro}

	\begin{proof}
Again we may assume that $\mathbf{0}\in \mathsf{A}$. If $|\mathsf{A}|=2$, Proposition \ref{p:OmegaA} and the last part of the proof of Theorem \ref{t:posent} can easily be adapted to show that $h(\Omega _\mathsf{A})=0$. If $|\mathsf{A}|\ge 3$, choose a primitive subgroup $\Gamma \subset \mathbb{Z}^d$ such that $\Gamma \cong \mathbb{Z}^2$ and $\Gamma \cap \mathsf{A}$ has at least three elements. A slight extension of the last part of the proof of Theorem \ref{t:posent} shows that $h(\Omega _\mathsf{A})>0$.
	\end{proof}

The second corollary of the proof of Theorem \ref{t:posent} concerns topological mixing of $\Omega _\mathsf{A},\,\mathsf{A}\subset \mathbb{Z}^d$. Recall that the \textit{SFT} $\Omega _\mathsf{A}\subset \mathsf{A}^{\mathbb{Z}^d}$ is \textit{mixing} if there exists, for any pair of nonempty finite sets $V,V'\subset \mathbb{Z}^d$, an $N \in \mathbb{N}$ with the following property: for every $\omega ,\omega '\in \Omega _\mathsf{A}$ and every $\mathbf{n}\in \mathbb{Z}^2$ with $\|\mathbf{n}\|\ge N$ there exists a $\omega ''\in \Omega _\mathsf{A}$ with $\omega ''|_V=\omega |_V$ and $\omega ''|_{V'+\mathbf{n}}=\omega '|_{V'+\mathbf{n}}$. Here $\| \cdot\|$ is the maximum norm on $\mathbb{Z}^d$, and $\omega |_W\in \mathsf{A}^W$ denotes the restriction of $\omega $ to its coordinates in a nonempty subset $W\subset \mathbb{Z}^d$.

	\begin{coro}
	\label{c:mixing}
If $\mathsf{D}=\mathsf{A}-\mathsf{A}$ is contained in a one-dimensional subspace of $\mathbb{R}^d$, then $\Omega _\mathsf{A}$ is not mixing.
	\end{coro}

	\begin{proof}
If $\mathsf{D}=\mathsf{A}-\mathsf{A}$ is contained in a one-dimensional subspace $V\subset \mathbb{R}^d$, and if $\mathbf{m}\in \mathsf{A}$, then there exists a primitive element $\mathbf{v}\in \mathbb{Z}^d$ so that $\mathsf{A}'=\mathsf{A}-\mathbf{m}\subset S\coloneqq \mathbb{Z}\mathbf{v}$. By Proposition \ref{p:OmegaA}, the \textit{SFT} $\Omega _{\mathsf{A}'}|_S$ is not mixing under the action of $\sigma ^\mathbf{v}$, and the last part of the proof of Theorem \ref{t:posent} shows that neither $\Omega _{\mathsf{A}'}$ nor $\Omega _\mathsf{A}$ can be mixing.
	\end{proof}

Proposition \ref{p:OmegaA} shows that $\Omega _A$ is nonmixing for every finite set $\mathsf{A}\subset \mathbb{Z}$ with at least two elements. For finite subsets $\mathsf{A}\subset \mathbb{Z}^d$ with $d>1$, the situation can be different, as the following examples show.

	\begin{exam}
	\label{e:topmix2}
Let $\mathsf{A}=\{(0,0),(1,0),(0,1)\} \subset \mathbb{Z}^2$. Then the $\mathbb{Z}^2$-\textit{SFT} $\Omega _\mathsf{A}$ is topologically mixing.

Since the the following verification of this claim will reappear --- in a slightly more complicated form --- in the proof of Theorem \ref{t:topmix}, we shall describe it in detail.

We start a bit of notation. Let $\mathsf{A}\subset \mathbb{Z}^d,\, d\ge1$, be a finite set containing $\mathbf{0}$. A subset $p\subset \mathbb{Z}^d$ is \textit{allowed} if it consists either of a single point or of a bi-infinite sequence $\{p_k\}_{k\in \mathbb{Z}}$ such that $p_{k+1}-p_k\in \mathsf{A}\smallsetminus \{\mathbf{0}\}$ for every $k\in \mathbb{Z}$. If $p\in \mathbb{Z}^d$ is allowed, its \textit{future} $p^+$ is defined by
	\begin{displaymath}
p^+=
	\begin{cases}
\{\mathbf{n}\}&\textup{if}\enspace p=\{\mathbf{n}\}\enspace \textup{for some}\enspace \mathbf{n}\in \mathbb{Z}^d,
	\\
\{p_k\}_{k\ge1} &\textup{if}\enspace p=\{p_k\}_{k\in \mathbb{Z}}\enspace \textup{with}\enspace p_{k+1}-p_k\in \mathsf{A}\smallsetminus \{\mathbf{0}\}\enspace \textup{for every}\enspace k\in \mathbb{Z}.
	\end{cases}
	\end{displaymath}

The definition of the \textit{past} $p^-$ of $p$ is analogous.

\smallskip A collection $\mathsf{p}$ of disjoint subsets of $\mathbb{Z}^d$ is \textit{allowed} if it consists of allowed sets. More generally, if $S\subset \mathbb{Z}^d$ is a subset, a collection $\mathsf{q}$ of disjoint subsets of $S$ is \textit{allowed} if there exists an allowed collection $\mathsf{p}$ of disjoint subsets of $\mathbb{Z}^d$ such that $\mathsf{q}=\mathsf{p}\cap S=\{p\cap S:p\in \mathsf{p}\}$.

\smallskip Every permutation $\pi \in \Pi _{\mathsf{A}}$ can be represented by the allowed partition $\mathsf{p}^{(\pi )}$ of $\mathbb{Z}^d$ into orbits or `paths' of $\pi $. Conversely, if $\mathsf{p}$ is an allowed collection of disjoint subsets of $\mathbb{Z}^d$, we denote by $\mathsf{U}_\mathsf{p}=\bigcup_{p\in \mathsf{p}}p$ the \textit{union} of $\mathsf{p}$ and extend $\mathsf{p}$ to an allowed partition $\tilde{\mathsf{p}}$ of $\mathbb{Z}^2$ by adding to it the singletons $\{\mathbf{n}\},\,\mathbf{n}\in \mathbb{Z}^2\smallsetminus \mathsf{U}_{\mathsf{p}}$. The partition $\tilde{\mathsf{p}}$ is the set of orbits of a unique permutation $\pi _{\tilde{\mathsf{p}}}\in \Pi _\mathsf{A}$.

\smallskip We return to the above set $\mathsf{A}=\{(0,0),(1,0),(0,1)\}$. Take a finite set $Q_1\subset \mathbb{Z}^2$ and a permutation $\pi_1 \in \Pi_\mathsf{A}$, and consider the collection $\mathsf{p}^{(\pi_1)}(Q_1)=\{p\in \mathsf{p}^{(\pi _1)}:p\cap Q_1\ne \varnothing \}$ of all orbits in the partition $\mathsf{p}^{(\pi _1)}$ which pass through $Q_1$.\label{paths1} We modify the orbits $p\in \mathsf{p}^{(\pi _1)}(Q_1)$ outside $Q_1$ in such a way that they are still allowed and move almost all the time vertically. Denote this family of modified orbits by $\mathsf{p}_1'(Q_1)$. Then we do the same for another permutation $\pi_2 \in \Pi_\mathsf{A}$ and another finite set $Q_2\subset \mathbb{Z}^2$ with sufficient horizontal distance from $Q_1$, and obtain a modified allowed collection $\mathsf{p}_2'(Q_2)$. Since $Q_1$ and $Q_2$ have sufficient horizontal distance, $\mathsf{U}_{\mathsf{p}_1'(Q_1)} \cap \mathsf{U}_{\mathsf{p}_2'(Q_2)}=\varnothing $ and the union $\mathsf{p}'=\mathsf{p}_1'(Q_1) \cup \mathsf{p}_2'(Q_2)$ is again an allowed collection of orbits. Finally, we extend $\mathsf{p}'$ to an allowed partition $\tilde{\mathsf{p}}'$ of $\mathbb{Z}^2$ by adding singletons and obtain a $\pi '=\pi _{\tilde{\mathsf{p}}'}\in \Pi_\mathsf{A}$ which coincides on $Q_1$ and $Q_2$ with $\pi_1$ and $\pi_2$, respectively.

If the sets $Q_1$ and $Q_2$ are separated vertically rather than horizontally, the modification process described above has to be changed accordingly.

\smallskip Now the details: let $M\ge 0$, put $Q=Q^{(M)}=\{0,\dots ,M\}^2\subset \mathbb{Z}^2$, and set $\mathbb{E}_k=\{0,\dots ,M+k\}\times \mathbb{N} \subseteq \mathbb{Z}^2$. Fix $\pi _1\in \Pi _\mathsf{A}$. The collection $\mathsf{p}\coloneqq \mathsf{p}^{(\pi _1)}(Q)$ of all $\pi_1$-orbits intersecting $Q$ is finite: it has at most $|Q|$ elements.

We start an induction process by setting $\mathsf{p}^{(0)}=\mathsf{p}$ and $\mathsf{q}^{(0)}=\varnothing $. Suppose that $K\ge0$, and that we have found allowed collections of disjoint sets $\mathsf{p}^{(k)}\subset \mathsf{p}$ and $\mathsf{q}^{(k)}$, which satisfy the following conditions for $k=0,\dots ,K$:
	\begin{enumerate}
	\item[(i)]
$q^+\subset \mathbb{E}_k$ for every $q\in \mathsf{q}^{(k)}$,
	\item[(ii)]
$E_{p,q}\coloneqq (p\cap \tilde{Q}) \cap (q\cap \tilde{Q})=\varnothing $ for every $p\in \mathsf{p}^{(k)}$ and $q\in \mathsf{q}^{(k)}$, where $\tilde{Q}=Q\cup \pi _1(Q)$,
	\item[(iii)]
The sets $\{p\cap \tilde{Q}:p\in \mathsf{p}^{(k)}\}\cup \{q\cap \tilde{Q}:q\in \mathsf{q}^{(k)}\}$ form a partition of $\tilde{Q}$ which coincides with $\{p\cap \tilde{Q}:p\in \mathsf{p}\}$.
	\end{enumerate}
For the induction step we suppose that $\mathsf{p}^{(K)}\ne \varnothing $ and set $\mathsf{p}_1^{(K)}=\{p\in \mathsf{p}^{(K)}:p^+\subset \mathbb{E}_{K+1}\}$ and $\mathsf{p}_2^{(K)}=\mathsf{p}^{(K)}\smallsetminus \mathsf{p}_1^{(K)}$. If $\mathsf{p}_1^{(K)}\ne\varnothing $, put $\mathsf{p}^{(K+1)}=\mathsf{p}_2^{(K)}$. If $\mathsf{p}_1^{(K)}=\varnothing $, every $p\in \mathsf{p}^{(K)}$ will move into $\mathbb{E}_{K+2}$ after having passed through $\mathbb{E}_{K+1}\smallsetminus \mathbb{E}_K$ (this follows from the fact that every infinite orbit of $\pi _1$ can only move up or right in steps of size one). Since $\mathsf{p}$ is finite, and since $p^+\cap \mathbb{E}_{K+1}$ is finite for every $p \in \mathsf{p}^{(K)}$ by assumption, the set $F=\bigcup_{p\in \mathsf{p}^{(K)}}p^+\cap (\mathbb{E}_{K+1}\smallsetminus \mathbb{E}_K)$ is finite and contains at least one element $\mathbf{m}$ whose second coordinate is maximal (i.e., satisfies that $\mathbf{m}+l\mathbf{e}^{(2)}\notin F$ for every $l>0$). We denote by $p'=\{p'_k\}_{k\in \mathbb{Z}}\in \mathsf{p}^{(K)}$ the element containing $\mathbf{m}$ with $p_{k_0}'=\mathbf{m}$, say, and define another allowed set $p''=\{p_k''\}_{k\in \mathbb{Z}}$ by setting
	\begin{displaymath}
p_k''=
	\begin{cases}
p_k'&\textup{if}\enspace k\le k_0,
	\\
p_{k_0}'+(k-k_0)\mathbf{e}^{(2)}&\textup{if}\enspace k>k_0.
	\end{cases}
	\end{displaymath}
Put $\mathsf{p}^{(K+1)}=\mathsf{p}^{(K)}\smallsetminus \{p'\}$ and $\mathsf{q}^{(K+1)}=\mathsf{q}^{(K)}\cup \{p''\}$. By assumption, $\mathsf{q}^{(K+1)}$ is again allowed.

In either case, $\mathsf{p}^{(K+1)}\subsetneq \mathsf{p}^{(K)}$, and the sets $\mathsf{p}^{(k)},\mathsf{q}^{(k)}$, $k=0,\dots ,K+1$, satisfy the condition (i) -- (iii) above with $K+1$ replacing $K$.

Since $\mathsf{p}$ is finite and $\mathsf{p}^{(k+1)}\subsetneq \mathsf{p}^{(k)}$ for every $k\ge0$, there has to exist a $K\ge1$ with $K\le |Q|$ such that $\mathsf{p}^{(K)}=\varnothing $, and hence with $q^+\subset \mathbb{E}_{K+1}$ for every $q\in \mathsf{q}^{(K)}$.

We have arrived at an allowed collection $\mathsf{q}=\mathsf{q}^{(K)}$ of disjoint subsets of $\mathbb{Z}^2$ such that $q^+\subset \mathbb{E}_{K+1}$ for every $q\in \mathsf{q}$ and the partitions $\{q\cap \tilde{Q}:q\in \mathsf{q}\}$ and $\{p\cap \tilde{Q}:p\in \mathsf{q}\}$ coincide. We extend $\mathsf{q}$ to an allowed partition $\tilde{\mathsf{q}}$ of $\mathbb{Z}^2$ by adding singletons and obtain a permutation $\tilde{\pi }_1\coloneqq \pi _{\tilde{\mathsf{q}}}\in \Pi _\mathsf{A}$ with the properties that $\tilde{\pi }_1^k(\mathbf{n}) = \pi _1^k(\mathbf{n})$ for every $\mathbf{n}\in Q$ and $k\le 1$, and that $\tilde{\pi }_1^k(\mathbf{n})\in \mathbb{E}_{|Q|+1}$ for every $\mathbf{n}\in Q$ and $k\ge0$.

Exactly the same argument, but with directions reversed, allows us to find a permutation $\Tilde{\tilde{\pi }}_1\in \Pi _\mathsf{A}$ such that $\Tilde{\tilde{\pi }}_1^k(\mathbf{n}) = \tilde{\pi }_1^k(\mathbf{n})$ for every $\mathbf{n}\in Q$ and $k\ge -1$, and that $\Tilde{\tilde{\pi }}_1^k(\mathbf{n})\in \{-|Q|-1,\dots ,M\}\times (-\mathbb{N})$ for every $\mathbf{n}\in Q$ and $k\le0$.

The permutation $\Tilde{\tilde{\pi }}_1$ has the properties that $\Tilde{\tilde{\pi }}_1|_Q=\tilde{\pi }_1|_Q=\pi _1|_Q$, and that each orbit of $\Tilde{\tilde{\pi }}_1'$ lies in the vertical strip $\{-|Q|-1,\dots ,M+|Q|+1\}\times \mathbb{Z}$.

An analogous argument yields a permutation $\Tilde{\tilde{\pi }}_1'\in \Pi _\mathsf{A}$ such that $\Tilde{\tilde{\pi }}_1'|_Q=\pi _1|_Q$, and that each orbit of $\Tilde{\tilde{\pi }}_1$ lies in the horizontal strip $\mathbb{Z}\times \{-|Q|-1,\dots ,M+|Q|+1\}$.

\smallskip By translating this back to the shift space $\Omega _\mathsf{A}$ we conclude that there exists, for every $M\ge 0$, every pair $\omega _1,\omega _2\in \Omega _\mathsf{A}$, and every $\mathbf{m}\in \mathbb{Z}^2$ with $\|\mathbf{m}\|>7|Q^{(M)}|$, an element $\omega _3\in \Omega _\mathsf{A}$ with $\omega _3|_{Q^{(M)}}=\omega _1|_{Q^{(M)}}$ and $\omega _3|_{Q^{(M)}+\mathbf{m}}=\omega _2|_{Q^{(M)}+\mathbf{m}}$. Clearly this implies that $\Omega _\mathsf{A}$ is mixing.
	\end{exam}

Example \ref{e:topmix2} illustrates the following general result.

	\begin{theo}
	\label{t:topmix}
Let $d\ge2$, and let $\mathsf{A}\subset \mathbb{Z}^d$ be a finite set. Then the $\mathbb{Z}^d$-\textit{SFT} $\Omega _\mathsf{A}$ is topologically mixing if and only if $\mathsf{D}=\mathsf{A}-\mathsf{A}$ does not lie in a one-dimensional subspace of $\mathbb{R}^d$.
	\end{theo}

	\begin{proof}
We start the proof of Theorem \ref{t:topmix} with a simplification and a bit of notation. As we observed in the proof of Theorem \ref{t:posent}, we may assume without loss in generality that $\mathbf{0}\in \mathsf{A}$; in fact, we shall assume that $\mathbf{0}$ is a vertex of the closed convex hull $\bar{\mathsf{A}}$ of $\mathsf{A}$ in $\mathbb{R}^d$.

Write $C(\mathsf{A})=\bigl\{\sum_{\mathbf{n}\in \mathsf{A}}t_\mathbf{n}\mathbf{n}: t_\mathbf{n}\ge0\enspace \textup{for every}\enspace \mathbf{n}\in \mathsf{A}\bigr\}\subset \mathbb{R}^d$ for the \textit{cone} of $\mathsf{A}$. Let $\mathbf{e}\in \mathsf{A}$ be another vertex of $\bar{\mathsf{A}}$ such that the ray $\{t\mathbf{e}:t\ge0\}\subset C(\mathsf{A})$ is extremal, and let $\mathsf{E}=\{t\mathbf{e}:t\in \mathbb{R}\}$. Put $\mathsf{B}=\mathsf{A}\smallsetminus \mathsf{E}$ and denote by $\bar{\mathsf{B}}$ the convex hull of $\mathsf{B}$. Then $\mathsf{E}$ and $\bar{\mathsf{B}}$ are disjoint convex subsets of $\mathbb{R}^d$, one of which is compact, and the strict hyperplane separation theorem (cf. \cite[Section 2.5.1]{BV} or \cite{Klee}) implies that there exist a vector $\mathbf{w}\in \mathbb{R}^d$ and real numbers $c_1<c_2$ such that $\langle \mathbf{w},\mathbf{v}\rangle \le c_1$ and $\langle \mathbf{w},\mathbf{v}'\rangle \ge c_2$ for every $\mathbf{v}\in \mathsf{E}$ and $\mathbf{v}'\in \bar{\mathsf{B}}$, where $\langle \cdot ,\cdot \rangle $ denotes the usual scalar product in $\mathbb{R}^d$. Since the first inequality holds for every $\mathbf{v}\in \mathsf{E}$ we conclude that $\langle \mathbf{w},\mathbf{e}\rangle =c_1=0$, whereas $\langle \mathbf{w},\mathbf{n}\rangle \ge c_2>0$ for every $\mathbf{n}\in \mathsf{B}$.

Now we can imitate the argument in Example \ref{e:topmix2}. Take a finite set $Q\subset \mathbb{Z}^d$ and set $L_1=\max_{\mathbf{n}\in Q}|\langle \mathbf{w},\mathbf{n}\rangle |$, $L_2=\max_{\mathbf{n}\in \mathsf{B}}\langle \mathbf{w},\mathbf{n}\rangle +1$, and $\mathbb{E}_k=\{\mathbf{k}\in \mathbb{Z}^d:-L_1\le \langle \mathbf{w},\mathbf{k}\rangle \le L_1+kL_2\}$, $k\ge0$. Then the induction argument in Example \ref{e:topmix2} can be used essentially unchanged, apart from replacing $\mathbf{e}^{(2)}$ with $\mathbf{e}$.

As there, we start with an element $\pi _1\in \Pi _\mathsf{A}$, consider the collection $\mathsf{p}\coloneqq \mathsf{p}^{(\pi _1)}(Q)$ of all $\pi_1$-orbits intersecting $Q$, and end up with an allowed collection $\mathsf{q}$ of disjoint subsets of $\mathbb{Z}^d$ and a $K\ge0$ such that $q^+\subset \mathbb{E}_{K}$ for every $q\in \mathsf{q}$ and the partitions $\{q\cap \tilde{Q}:q\in \mathsf{q}\}$ and $\{p\cap \tilde{Q}:p\in \mathsf{p}\}$ coincide, where $\tilde{Q}=Q\cup \pi _1(Q)$. Again we extend $\mathsf{q}$ to an allowed partition $\tilde{\mathsf{q}}$ of $\mathbb{Z}^d$ by adding singletons and obtain a permutation $\tilde{\pi }_1\coloneqq \pi _{\tilde{\mathsf{q}}}\in \Pi _\mathsf{A}$ with the properties that $\tilde{\pi }_1^k(\mathbf{n}) = \pi _1^k(\mathbf{n})$ for every $\mathbf{n}\in Q$ and $k\le 1$, and that $\tilde{\pi }_1^k(\mathbf{n})\in \mathbb{E}_{K}$ for every $\mathbf{n}\in Q$ and $k\ge0$. This takes care of forward orbits; the backward orbits of $\pi _1$ are dealt with by reversing directions as in Example \ref{e:topmix2}.

Finally we note that we can replace the vertex $\mathbf{e}\in \mathsf{A}$ by any other vertex $\mathbf{e}'\in \mathsf{A}$ giving rise to an extremal ray of $C(\mathsf{A})$ (here we are using our assumption that $\mathsf{D}=\mathsf{A}-\mathsf{A}$ is not one-dimensional). This allows us to complete the proof of Theorem \ref{t:topmix} as in Example \ref{e:topmix2}.
	\end{proof}

	\begin{rems}
	\label{r:entropy}
(1) The dynamical properties of $\Omega _\mathsf{A}$ discussed here, like topological mixing or topological entropy, are \textit{affine} invariants: for every finite set\/ $\mathsf{A}\subset \mathbb{Z}^d$ and every $\gamma \in \mathbb{Z}^d \rtimes \textup{GL}(d,\mathbb{Z})$, $\Omega _{\gamma \mathsf{A}}$ is obtained from $\Omega _\mathsf{A}$ through an affine re-parametrization of the coordinates. In particular, $\Omega _{\gamma \mathsf{A}}$ is mixing if and only if the same is true for $\Omega _\mathsf{A}$, and $h(\Omega _{\gamma \mathsf{A}})=h(\Omega _\mathsf{A})$.

\smallskip (2) Permutations with restricted movement can be considered in a more general context than we have done here. Let $\Gamma $ be a countable discrete group $\Gamma $, and let $\mathsf{A}\subset \Gamma $ be a nonempty finite set. Consider the set $\Pi _\textsf{A}\subset S^\infty (\Gamma )$ of all permutations $\pi \colon \gamma \mapsto \pi (\gamma )$ such that
	\begin{equation}
	\label{eq:PiII}
\omega _\gamma ^{(\pi )}\coloneqq \pi (\gamma )\gamma ^{-1} \in \mathsf{A} \enspace \textup{for every}\enspace \gamma \in \Gamma .
	\end{equation}
We put
	\begin{equation}
	\label{eq:OmegaF}
\Omega _\mathsf{A}=\{\omega ^{(\pi )}:\pi \in \Pi _\mathsf{A}\}
	\end{equation}
and observe as for $\Gamma =\mathbb{Z}$ that $\Omega _\mathsf{A}\subset \mathsf{A}^\Gamma $ is a shift of finite type for the right shift-action $\sigma $ of $\Gamma $ on $\mathsf{A}^\Gamma $, defined by
	\begin{equation}
	\label{eq:rightshift}
(\sigma ^\delta \omega )_\gamma =\omega _{\gamma \delta }
	\end{equation}
for every $\omega =(\omega _\gamma )_{\gamma \in \Gamma }\in \mathsf{A}^\Gamma $ and $\delta \in \Gamma $. This construction gives rise to a natural class of examples of $\Gamma $-\textit{SFT}'s for any countable discrete group $\Gamma $.
	\end{rems}

\section{Example \ref{e:topmix2}, revisited}\label{s:example}

As an illustration of properties of the multiparameter \textit{SFT}'s appearing in Section \ref{s:Z2} we return to the $\mathbb{Z}^2$-\textit{SFT} $\Omega _\mathsf{A}$, $\mathsf{A}=\{(0,0),(1,0),(0,1)\}$, in Example \ref{e:topmix2}. As described in Example \ref{e:topmix2}, every $\omega \in \Omega _\mathsf{A}$ determines a permutation $\pi ^{(\omega )}$ of $\mathbb{Z}^2$, each orbit of which consists either of a single point or of a bi-infinite sequence $(\mathbf{n}_k)_{k\in \mathbb{Z}}$ with $\mathbf{n}_{k+1}-\mathbf{n}_k\in \{(1,0),(0,1)\}$ for every $k\in \mathbb{Z}$. If we represent each infinite orbit of $\pi ^{(\omega )}$ by a bi-infinite directed polygonal path in $\mathbb{Z}^2$, we obtain a collection $\mathsf{p}^{(\pi ^{(\omega )})}$ of non-intersecting paths in $\mathbb{Z}^2$ moving either north or east at each step. Figure \ref{f:figure8} shows the intersection of $\mathsf{p}^{(\pi ^{(\omega )})}$ with a square $Q\subset \mathbb{Z}^2$. In the terminology of Example \ref{e:topmix2}, $\mathsf{q}=\mathsf{p}^{(\pi ^{(\omega )})}\cap Q$ is an \textit{allowed collection} of disjoint subsets of $Q$ or, for convenience, an \textit{allowed configuration of paths} in $Q$. Conversely, every allowed configuration of paths in $Q$ arises in this manner from some element of $\Omega _\mathsf{A}$.\vspace{-5mm}

	\begin{figure}[ht] \includegraphics[width=45mm]{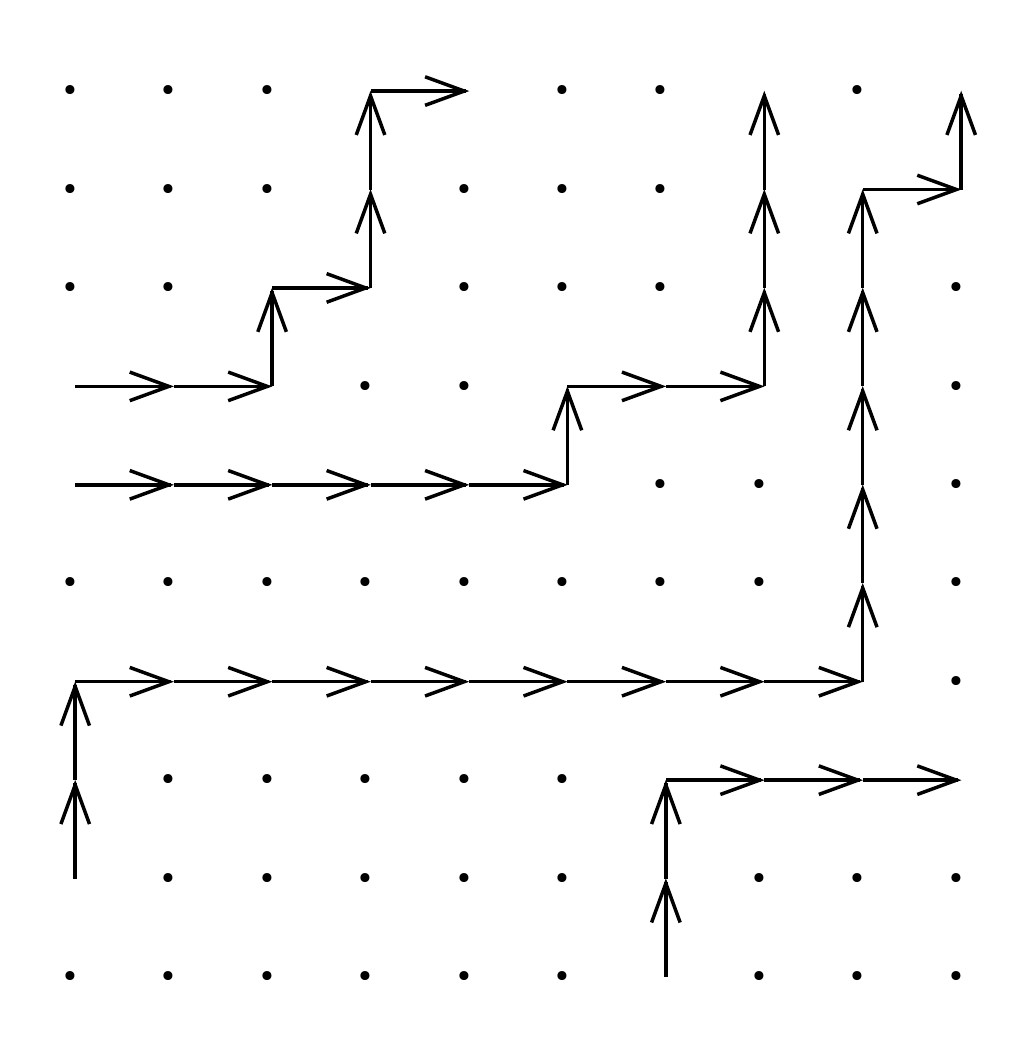} \vspace{-5mm}
\caption{An allowed configuration of paths in a square $Q$}\label{f:figure8}
	\end{figure}

The entropy of $\Omega _\mathsf{A}$ is positive by Theorem \ref{t:posent}. We are grateful to Christian Krattenthaler for pointing out and explaining to us \cite[Theorem 3.1]{Stembridge}, which yields an explicit formula for the number of allowed configurations of $k$ paths leading from the bottom and left edges to the top and right edges of the square $Q$ in Figure \ref{f:figure8}. By using this formula and varying both $k$ and the size $Q$ one obtains that the topological entropy of $\Omega _\mathsf{A}$ is about $\log 1.38\dots$

\smallskip According to Theorem \ref{t:topmix}, the \textit{SFT} $\Omega _\mathsf{A}$ is topologically mixing. However, it does not have the \textit{uniform filling property} (\cite[Definition 3.1]{Robinson-Sahin}), nor is it \textit{strongly irreducible} in the sense of \cite[Definition 1.10]{Burton-Steif}. The following proposition shows that $\Omega _A$ nevertheless has an abundance of periodic points, allowing to express its entropy in terms of the logarithmic growth rate of the number of its periodic points.

\smallskip For every finite-index subgroup $\Gamma \subset \mathbb{Z}^2$ denote by $\textup{Fix}_{\Gamma }(\Omega _\mathsf{A})=\{\omega \in \Omega _\mathsf{A}:\sigma ^\mathbf{n}\omega =\omega \enspace \textup{for every}\enspace \mathbf{n}\in \Gamma \}$ the set of $\Gamma $-periodic points in $\Omega _\mathsf{A}$. Then the following is true.

	\begin{prop}
	\label{p:periodic}
The set of periodic points is dense in $\Omega _\mathsf{A}$. Furthermore,
	\begin{equation}
	\label{eq:periodic}
h(\Omega _\mathsf{A})= \lim_{K\to\infty }\,\frac{1}{|\mathbb{Z}^2/\Delta _K|}\cdot \log\bigl|\textup{Fix}_{\Delta _K}(\Omega _\mathsf{A})\bigr|,
	\end{equation}
where $\Delta _K=\{2k(K^3+2K)\cdot (1,0)+2lK\cdot (1,1):k,l\in \mathbb{Z}\}\subset \mathbb{Z}^2$ for every $K\ge1$.
	\end{prop}

	\begin{proof}
Figure \ref{f:figure3} shows two allowed configurations $\mathsf{q}$ and $\mathsf{p}$ of paths in a polygon $R\subset \mathbb{R}^2$ with edges $a,b,c,d,e,f$ of lengths $|a|=|d|$ and $|b|=|c|=|e|=|f|$, whose intersections with the boundary $\partial R$ of $R$ (i.e., with the union of the edges of $R$) coincide.\vspace{-5mm}

	\begin{figure}[ht]\centering
\subfloat[The configuration \textsf{q}]{\includegraphics[width=75mm]{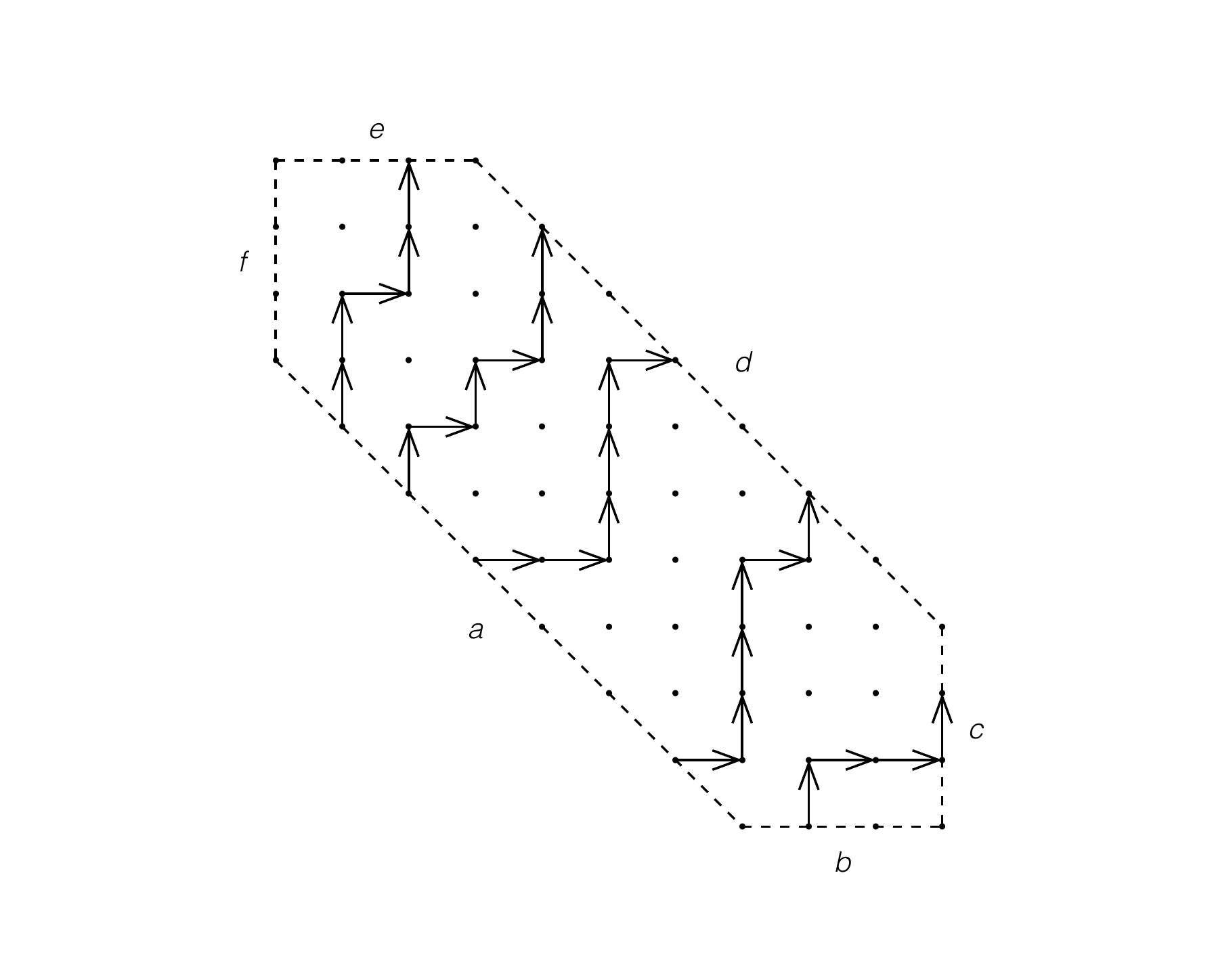}} \hspace{-20mm}
\subfloat[The configuration \textsf{p}]{\includegraphics[width=75mm]{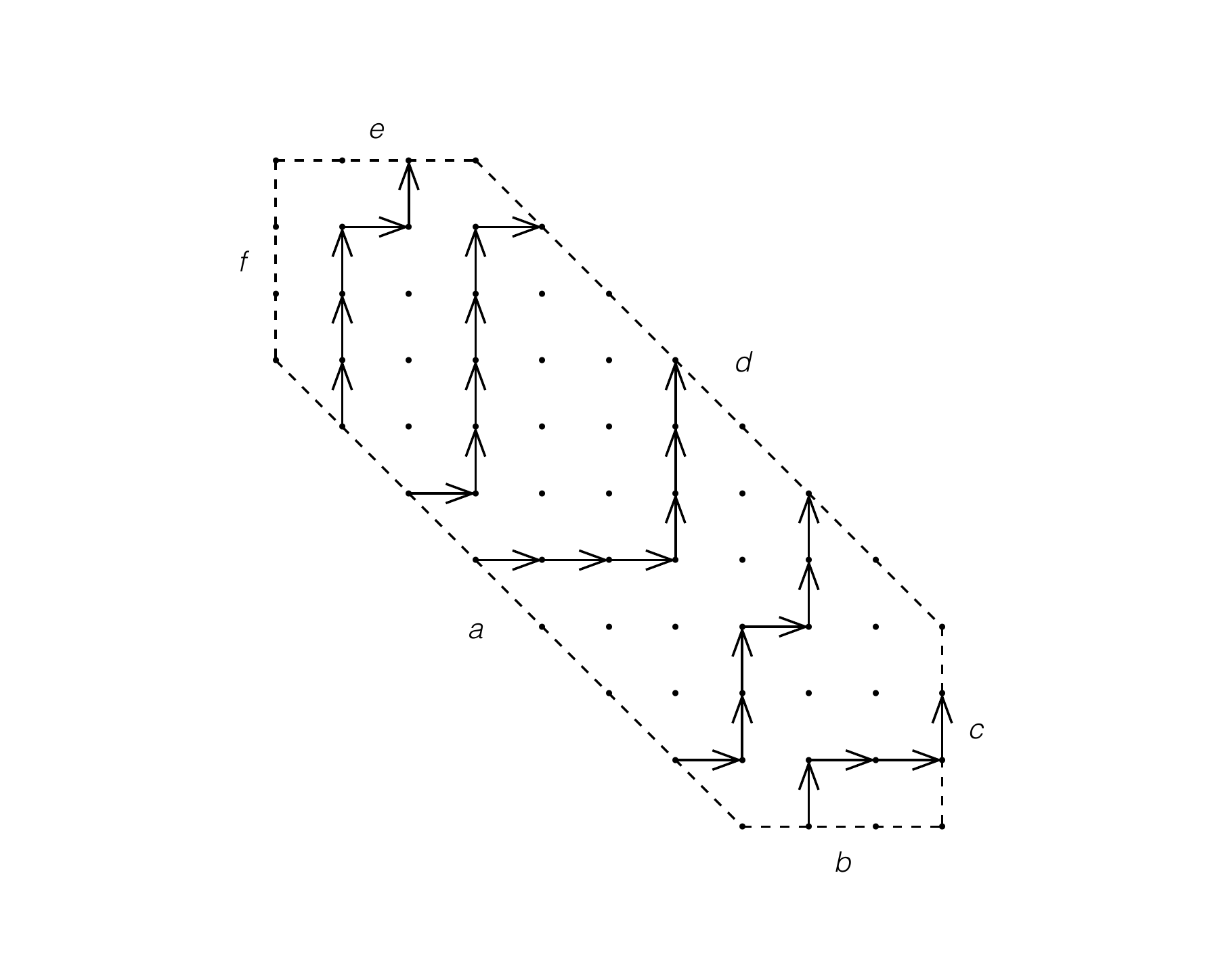}
}\caption{Two allowed configurations of paths in a polygon $R$}\label{f:figure3}
	\end{figure}

If we rotate the configuration $\mathsf{p}$ in Figure \ref{f:figure3} clockwise by 90$^\circ$, flip the resulting pattern horizontally, and reverse the direction of all arrows, we obtain another allowed configuration $\mathsf{p}'$ of paths in $R$, as shown in Figure \ref{f:figure4}.

	\begin{figure}[ht]
\includegraphics[width=80mm]{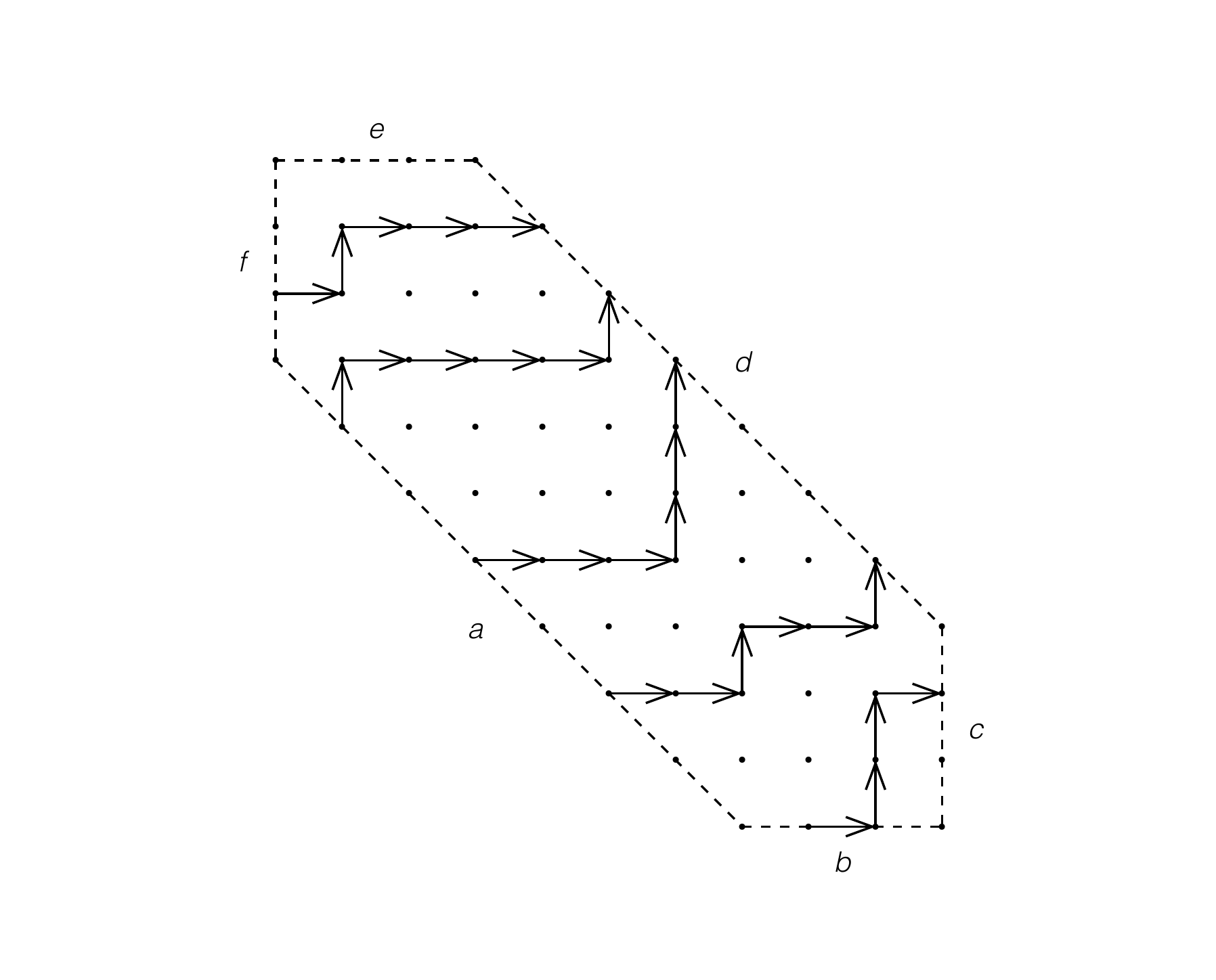}\vspace{-5mm}\caption{The configuration $\mathsf{p}'$.}\label{f:figure4}
	\end{figure}

In Figure \ref{f:figure5} we glue together the configurations $\mathsf{q}$ and $\mathsf{p}'$ along the edges labelled $d$ and $a$, respectively, put in `dotted' arrows to connect up loose ends, and obtain an allowed configuration $(\mathsf{q},\mathsf{p})$ of paths in a bigger polygon $\tilde{R}\subset \mathbb{Z}^2$.%\vspace{-6mm}

	\begin{figure}[ht]
\includegraphics[width=65mm]{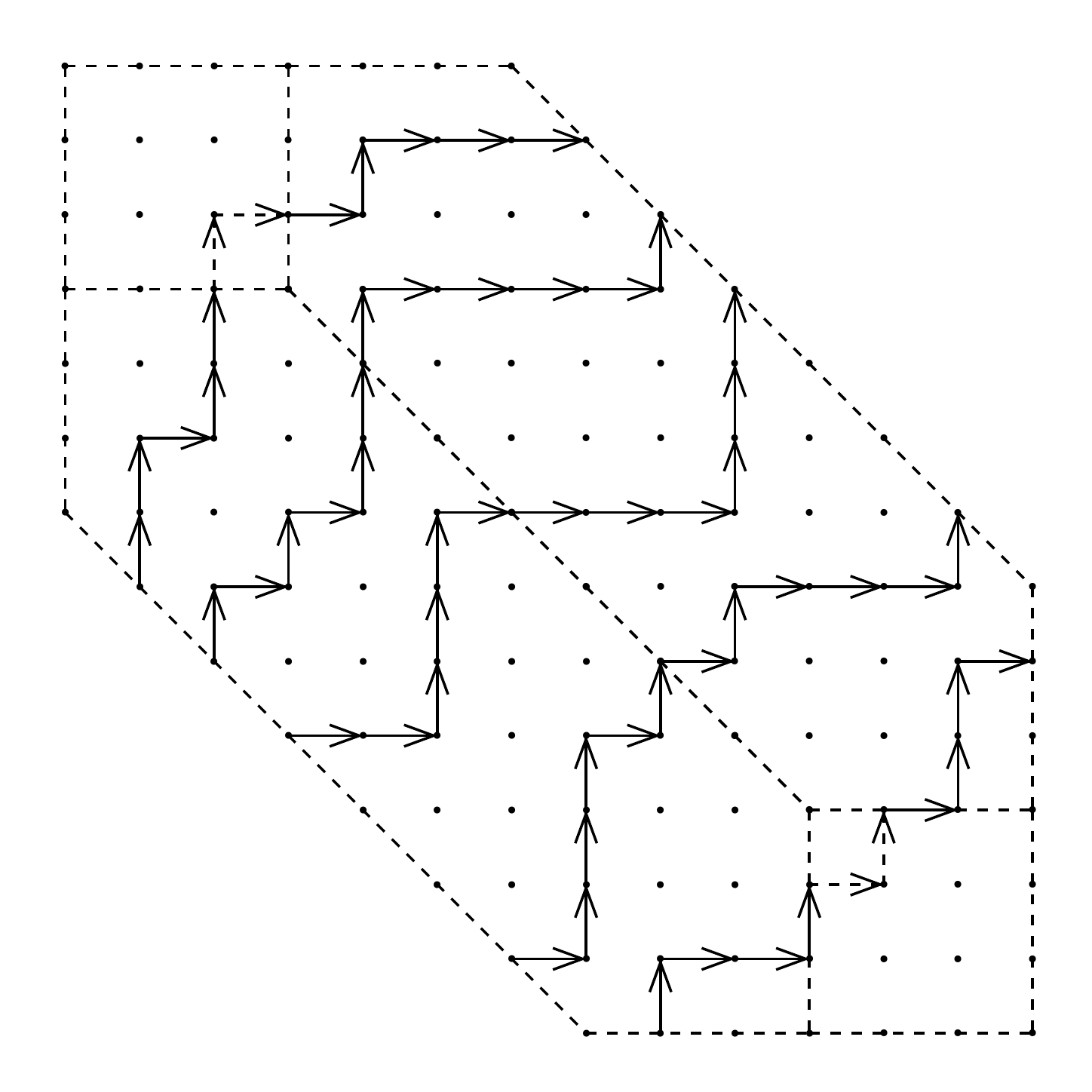}\vspace{-4mm}\caption{The configuration $(\mathsf{q},\mathsf{p})$ of paths in the bigger polygon $\tilde{R}$.}\label{f:figure5}
	\end{figure}

In Figure \ref{f:figure6} we extend the configuration $(\mathsf{q},\mathsf{p})$ in Figure \ref{f:figure5} periodically, again putting in dotted arrows to connect loose ends.

	\begin{figure}[ht]
\includegraphics[width=80mm]{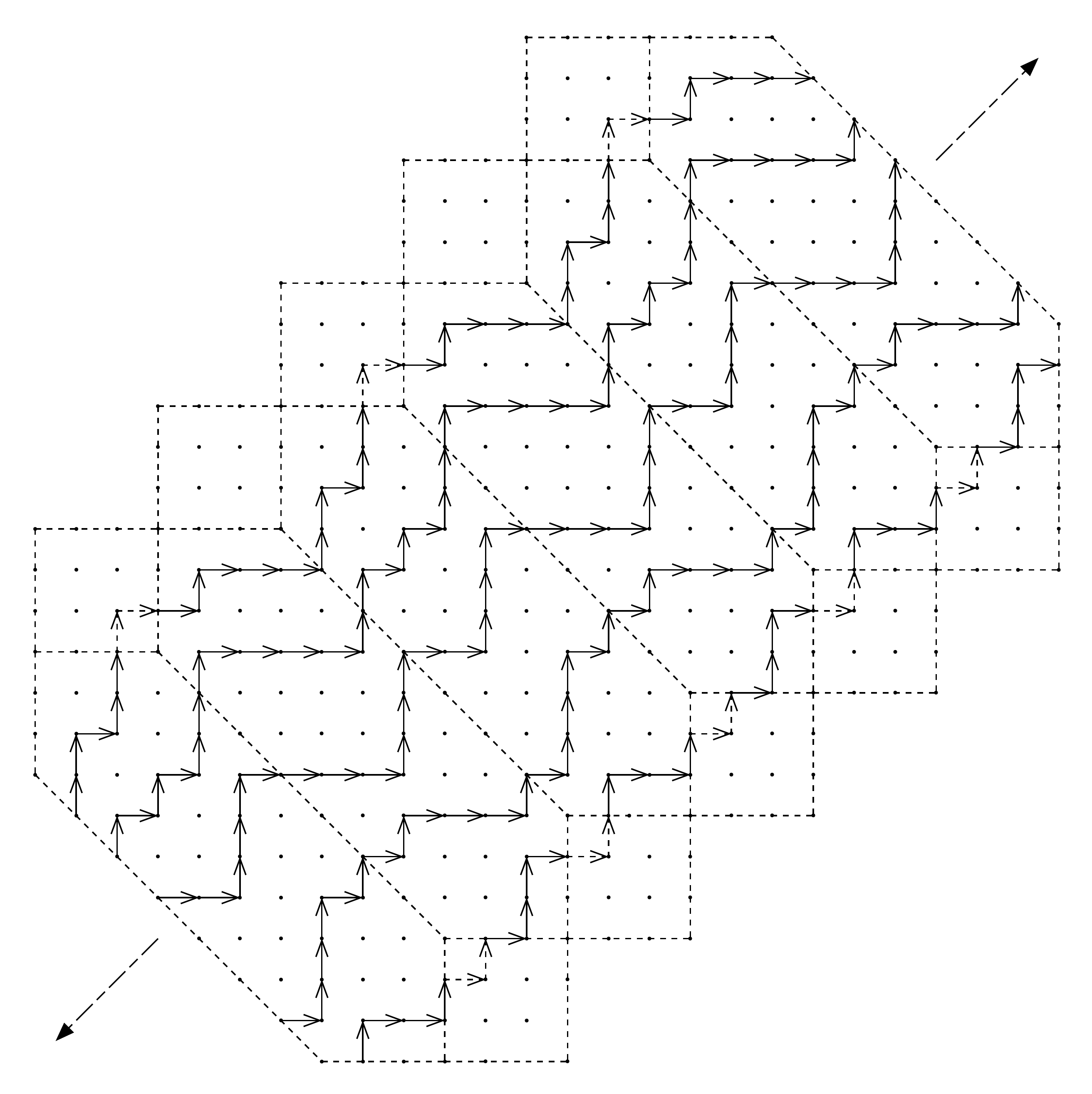}\vspace{-2mm} 	\caption{The periodic extension of the configuration $(\mathsf{q},\mathsf{p})$ in Figure \ref{f:figure5}.}\label{f:figure6}
	\end{figure}

Finally we repeat the diagonal strip in Figure \ref{f:figure6} periodically in the horizontal direction, allowing sufficient separation between the strips. Figure \ref{f:figure7} shows the repetition of Figure \ref{f:figure6} with a horizontal period of size $(|a|+2|b|\sqrt2)\sqrt2$ (the length of the hypothenuse of the triangle $ABC$).%\vspace{-6mm}

	\begin{figure}[ht]
\includegraphics[width=135mm]{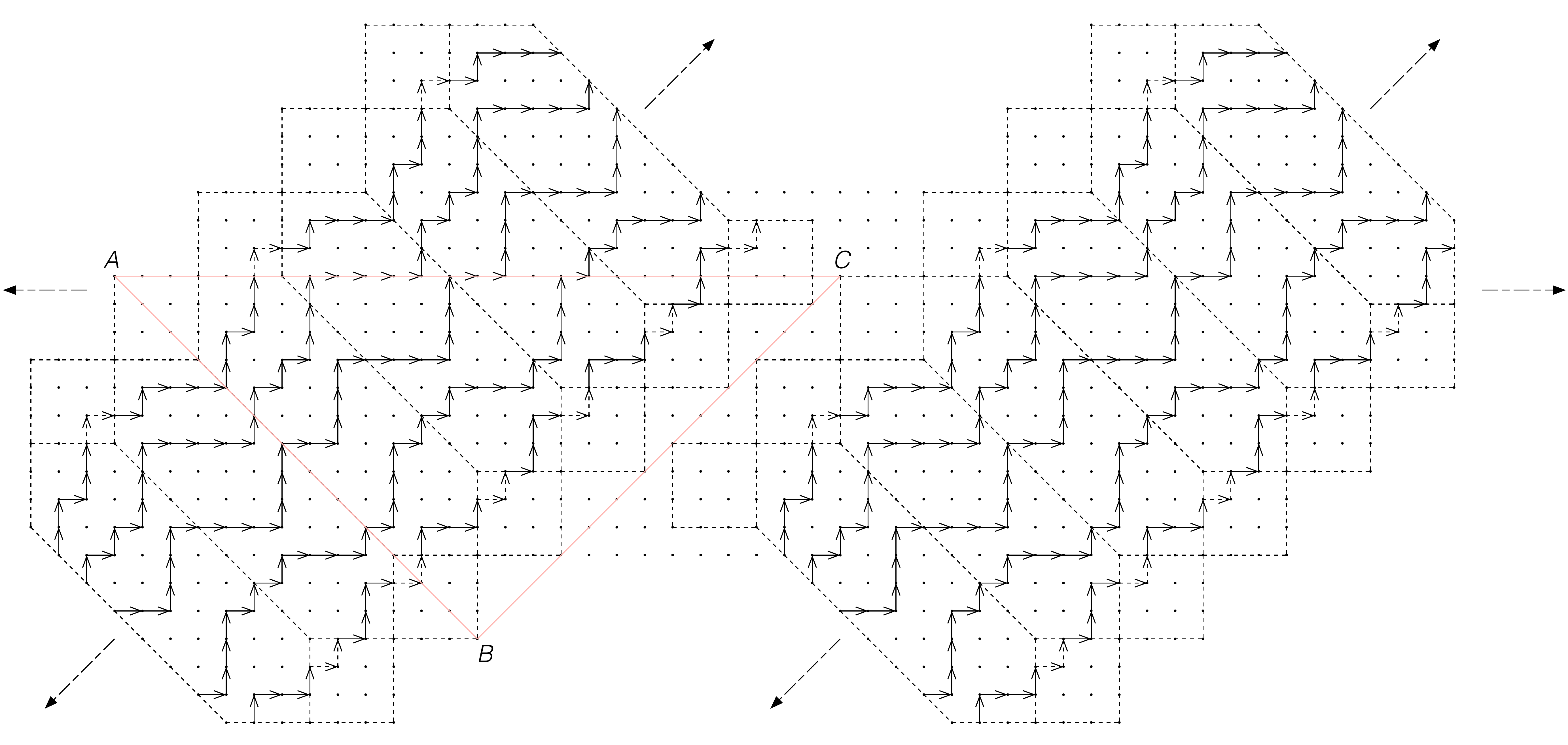}\vspace{-2mm} 	\caption{The periodic repetition of the diagonal strip in Figure \ref{f:figure6}.}\label{f:figure7}
	\end{figure}

\smallskip If we assume that the edges $a$ and $b$ of the polygon $R\subset \mathbb{Z}^2$ in Figure \ref{f:figure3} have lengths $|a|=K^3\sqrt2$ and $|b|=K$ and write $R_K$ instead of $R$ to emphasise the dependence of this polygon on the integer $K$, then the periodic configuration of paths in Figure \ref{f:figure7} corresponds to an element $\omega '\in \textup{Fix}_{\Delta _K}(\Omega _\mathsf{A})$, where
	\begin{equation}
	\label{eq:DeltaK}
	\begin{aligned}
\Delta _K&=\bigl\{k\bigl((|a|+2|b|\sqrt2)\sqrt2\bigr)(1,0) + 2l|b|(1,1):k,l\in \mathbb{Z}\bigr\}
	\\
&= \{2k(K^3+2K)(1,0) + 2lK(1,1):k,l\in \mathbb{Z}\}\subset \mathbb{Z}^2.
	\end{aligned}
	\end{equation}
If $\omega \in \Omega _\mathsf{A}$ is a point giving rise to the configuration $\mathsf{q}$ of paths in the polygon $R_K$, then the coordinates of $\omega $ and $\omega '$ coincide on $R_K$. Since $K$ was arbitrary, this proves the density of periodic points in $\Omega _A$.

For every $K\ge2$, the intersection $[R_K] \coloneqq R_K \cap \mathbb{Z}^2$ has cardinality $|[R_K]| = (2K+1)(K^3+1)+K^2$. The configuration $\mathsf{q}$ of paths in $R_K$ arising from an element $\omega \in \Omega _\mathsf{K}$ is determined by the coordinates $\omega _\mathbf{n},\,\mathbf{n}\in R_K$, of $\omega $. Conversely, $\mathsf{q}$ determines all the coordinates $\omega _\mathbf{n},\,\mathbf{n}\in R_K$, with the exception of (some of) the coordinates $\omega _\mathbf{n}$ with $\mathbf{n}$ lying on the boundary $\partial R_K$ of $R_K$ (cf. Figure \ref{f:figure3}). We write $N_K$ for the number of allowed configurations of paths in $R_K$ and conclude that
	\begin{equation}
	\label{eq:limits}
\lim_{K\to\infty }\frac{1}{|[R_K]|}\cdot \log N_K = \lim_{K\to\infty }\frac{1}{|[R_K]|}\cdot \log |\Pi _{[R_K]}(\Omega _\mathsf{A})| = h(\Omega _\mathsf{A}),
	\end{equation}
where we are using that $([R_K])_{K\ge1}$ is a F\o lner sequence in $\mathbb{Z}^2$, and where $\Pi _F\colon \Omega _\mathsf{A}\linebreak[0]\longrightarrow \mathsf{A}^F$ denotes the projection of every $\omega \in \Omega _\mathsf{A}$ onto its coordinates in a set $F\subset \mathbb{Z}^2$.

Since $\partial R_K$ intersects $\mathbb{Z}^2$ in $2K^3+4K \le 3K^3$ points, a set $\mathsf{C}$ of at least $M_K\coloneqq N_K/|\mathsf{A}|^{3K^3} = N_K/3^{3K^3}$ of these configurations must coincide on $\partial R_K$. By taking all pairs of elements in $\mathsf{C}$ we obtain at least $M_K^2=N_K^2/3^{6K^3}$ distinct allowed configurations $(\mathsf{q},\mathsf{p}),\,\mathsf{q},\mathsf{p}\in \mathsf{C}$, of paths in the polygon $\tilde{R}_K$ in Figure \ref{f:figure5}, each of which determines an element of $\textup{Fix}_{\Delta _K}(\Omega _\mathsf{A})$, where $\Delta _K\subset \mathbb{Z}^2$ is given in \eqref{eq:DeltaK} (cf. Figure \ref{f:figure7}). Hence $|\textup{Fix}_{\Delta _K}(\Omega _\mathsf{A})| \ge N_K^2/3^{6K^3}$. We set $[\tilde{R}_K]=\tilde{R}_K\cap \mathbb{Z}^2$ and note that
	\begin{displaymath}
|[\tilde{R}_K]| =(4K+1)(K^3+1)+4K^2 =2|[R_K]| +2K^2 -K^3-1.
	\end{displaymath}
Since $|\mathbb{Z}^2/\Delta _K| = 4K^4+8K^2$, we obtain that
	\begin{equation}
	\label{eq:period9}
	\begin{aligned}
\tfrac{1}{|\mathbb{Z}^2/\Delta _K|}\log |\textup{Fix}_{\Delta _K}(\Omega _\mathsf{A})| &= \tfrac{|[\tilde{R}_K]|}{4K^4+8K^2}\cdot \tfrac{2|[R_K]|}{|[\tilde{R}_K]|} \cdot \tfrac{1}{2|[R_K]|} \log |\textup{Fix}_{\Delta _K}(\Omega _\mathsf{A})|
	\\
&\ge\tfrac{4K^4}{4K^4+8K^2}\cdot \tfrac{2|[R_K]|}{|[\tilde{R}_K]|}\cdot  \tfrac{1}{|[R_K]|} \log (N_K/3^{3K^3})
	\\
&\ge\tfrac{4K^4}{4K^4+8K^2}\cdot \tfrac{2|[R_K]|}{|[\tilde{R}_K]|} \cdot \bigl(\tfrac{1}{|[R_K]|} \log N_K - \tfrac{3K^3}{2K^4}\log3 \bigr)
	\end{aligned}
	\end{equation}
for every $K\ge2$. By letting $K\to\infty $ and using \eqref{eq:limits} we obtain that
	\begin{displaymath}
\liminf_{K\to\infty } \tfrac{1}{|\mathbb{Z}^2/\Delta _K|}\log |\textup{Fix}_{\Delta _K}(\Omega _\mathsf{A})|\ge h(\Omega _\mathsf{A}).
	\end{displaymath}
Since the opposite inequality $\limsup_{K\to\infty } \frac{1}{|\mathbb{Z}^2/\Delta _K|}\log |\textup{Fix}_{\Delta _K}(\Omega _\mathsf{A})|\le h(\Omega _\mathsf{A})$ is obvious, we have proved \eqref{eq:periodic}.
	\end{proof}

	\begin{coro}
	\label{c:periodic2}
	$$
\limsup_{K\to\infty }\frac{1}{K^2}\log |\textup{Fix}_{K\mathbb{Z}^2}(\Omega _\mathsf{A})| = \lim_{K\to\infty }\frac{1}{16K^6}\log |\textup{Fix}_{4K^3\mathbb{Z}^2}(\Omega _\mathsf{A})|=h(\Omega _\mathsf{A}).
	$$
	\end{coro}

	\begin{proof}
For every $K\ge1$ we consider the Polygon $\tilde{R}_K$ appearing in Figure \ref{f:figure5} in the proof of Proposition \ref{p:periodic}. There we showed that there exists a set $\tilde{\mathsf{C}}$ of more than $N_K^2/3^{6K^3}$ distinct allowed configurations of paths in $\tilde{R}_K$ corresponding to points in $\textup{Fix}_{\Delta _K}(\Omega _\mathsf{A})$, all of which coincide on the boundary $\partial \tilde{R}_K$ of $\tilde{R}_K$.

We set $\Delta _K'=2(K^3+2K)\mathbb{Z}^2\subset \Delta _K$. Then $|\Delta _K/\Delta _K'|=K^2+2$, and $\Delta _K$ is the disjoint union of cosets $\Delta _K'+\mathbf{v},\,\mathbf{v}\in S_K$, with $|S_K|=K^2+2$. We choose, for every $\mathbf{v}\in S_K$, an arbitrary configuration $\mathsf{q}_\mathbf{v}\in \tilde{\mathsf{C}}$, fill the polygon $\tilde{R}_K+\mathbf{v}$ with the corresponding translate $\mathsf{q}_\mathbf{v}+\mathbf{v}$ of $\mathsf{q}_\mathbf{v}$, and obtain in this manner a family $\tilde{\mathsf{D}}$ of at least $(N_K^2/3^{6K^3})^{|\Delta _K/\Delta _K'|}$ distinct allowed configurations of paths in the set $F_K\coloneqq \bigcup_{\mathbf{v}\in S_K}\tilde{R}_K+\mathbf{v}$, each of which has a unique $\Delta _K'$-invariant extension to $\mathbb{Z}^2$ and determines an element of $\textup{Fix}_{\Delta _K'}(\Omega _\mathsf{A})$. From \eqref{eq:period9} we obtain that
	\begin{align*}
\tfrac{1}{|\mathbb{Z}^2/\Delta _K'|}&\log |\textup{Fix}_{\Delta _K'}(\Omega _\mathsf{A})| \ge \tfrac{1}{|\mathbb{Z}^2/\Delta _K|} \log N_K/3^{3K^3}
	\\
&\ge\tfrac{|[\tilde{R}_K]|}{4K^4+8K^2}\cdot \tfrac{2|[R_K]|}{|[\tilde{R}_K]|} \cdot \bigl(\tfrac{1}{|[R_K]|} \log N_K - \tfrac{3K^3}{2K^4}\log3 \bigr) \xrightarrow {K\to\infty } h(\Omega _\mathsf{A}).
	\end{align*}
Since $\limsup_{K\to\infty }\tfrac{1}{|\mathbb{Z}^2/\Delta _K'|}\log |\textup{Fix}_{\Delta _K'}(\Omega _\mathsf{A})|\le h(\Omega _\mathsf{A})$, this proves the corollary.
	\end{proof}

	\begin{pros}
	\label{p:entropy}
(1) Is $\lim_{K\to \infty }\frac{1}{K^2}\cdot \log\,|\textup{Fix}_{K\mathbb{Z}^2}(\Omega _\mathsf{A})|=h(\Omega _\mathsf{A})$?

\smallskip (2) Does $\Omega _\mathsf{A}$ have a unique shift-invariant probability measure of maximal entropy?

\smallskip (3) Let $\mu $ be \textit{a} (or \textit{the}) shift-invariant probability measure of maximal entropy on $\Omega _\mathsf{A}$. For every $\omega \in \Omega _\mathsf{A}$ and every $n\ge1$, consider the allowed configuration $\mathsf{p}^{(\omega )}(Q_n)\coloneqq \mathsf{p}^{(\pi ^{(\omega )})}\cap Q_n$ of paths in the square $Q_n=\{0,\dots ,n\}^2\subset \mathbb{Z}^2$ determined by $\omega $ (cf. Figure \ref{f:figure8}), and we write $N(\mathsf{p}^{(\omega )}(Q_n))$ for the number of paths (or connected components) of $\mathsf{p}^{(\omega )}(Q_n)$. Is it true that $\lim_{n\to\infty }\frac1n\cdot N(\mathsf{p}^{(\omega )}(Q_n))=\frac23$ for $\mu \textit{-a.e.}\,\omega \in \Omega _\mathsf{A}$, as numerical evidence suggests?

\smallskip (4) How general are the results in this section? Are analogous statements true for every finite set $\mathsf{A}\subset \mathsf{Z}^2$ such that $\Omega _\mathsf{A}$ is topologically mixing?
	\end{pros}

\end{document}